\documentclass{amsart}

\newtheorem{theorem}{Theorem}[section]
\newtheorem{Proposition}[theorem]{Proposition}
\newtheorem{lemma}[theorem]{Lemma}
\newtheorem{conjecture}[theorem]{Conjecture}
\newtheorem{Corollary}[theorem]{corollary}

\newtheorem{remark}[theorem]{Remark}

\numberwithin{equation}{section}

\usepackage{hyperref}

\begin{document}

\title{BERNOULLI OPERATOR AND RIEMANN'S ZETA FUNCTION}

\author{Yiping Yu}
 \email{yipingyu1@126.com}






\begin{abstract}
We introduce a Bernoulli operator,let $\mathbf{B}$ denote the operator symbol,for n=0,1,2,3,... let ${\mathbf{B}^n}: = {B_n}$ (where ${B_n}$ are Bernoulli numbers,${B_0} = 1,B{}_1 = 1/2,{B_2} = 1/6,{B_3} = 0$...).We obtain some formulas for Riemann's Zeta function,Euler constant and a number-theoretic function relate to Bernoulli operator.For example,we show that
\[{\mathbf{B}^{1 - s}} = \zeta (s)(s - 1),\]
\[\gamma  =  - \log \mathbf{B},\]where ${\gamma}$ is Euler constant.Moreover,we obtain an analogue of the Riemann Hypothesis (All zeros of the function $\xi (\mathbf{B} + s)$ lie on the imaginary axis).This hypothesis can be generalized to Dirichlet L-functions,Dedekind Zeta function,etc.In particular,we obtain an analogue of Hardy's theorem(The function $\xi (\mathbf{B} + s)$ has infinitely many zeros on the imaginary axis).
\par In addition,we obtain a functional equation of $\log \Pi (\mathbf{B}s)$ and a functional equation of $\log \zeta (\mathbf{B} + s)$ by using Bernoulli operator.
\end{abstract}

\maketitle



\section{\textbf{Introduction}}
We introduce a Bernoulli operator,let $\mathbf{B}$ denote the operator symbol, for n=0,1,2,3, ... let ${\mathbf{B}^n}: = {B_n}$ (where ${B_n}$ are Bernoulli numbers,${B_0} = 1,B{}_1 = 1/2,{B_2} = 1/6,{B_3} = 0$...).Despite the fact that Bernoulli defined $B{}_1 = 1/2$, some authors set $B{}_1 = -1/2$.In this paper,it will be convenient to set $B{}_1 = 1/2$.Using the operator,we are easy to obtain the equation ${(1 - \mathbf{B})^n} = {\mathbf{B}^n}$ (n is non-negative integer),and the following equation
\[{e^{ - \mathbf{B}z}} = \frac{z}{{{e^z} - 1}}.\]
\par First we introduce some definitions.
\par If a function's (equation's) expression contains the operator symbol "$\mathbf{B}$",then we say the function(equation) is a \textbf{Bernoulli operator function (equation)} or simply a function (equation).For example,the function $f(z)= e^{ - \mathbf{B}z}$ is a Bernoulli operator function.
\par If a Bernoulli operator function (equation) equal to another Bernoulli operator function (equation) without taking Bernoulli operator,we say the function (equation) is \textbf{primitive equivalent}.For example,the following equation is primitive equivalent:
\\$\sin (\pi \mathbf{B}/2) = \sin (\pi \mathbf{B}/2 + 2\pi ).$
\par If a Bernoulli operator function (equation) equal to another Bernoulli operator function (equation) by taking Bernoulli operator, we say the function (equation)is \textbf{Bernoulli operator equivalent}.For example,the following equation is Bernoulli operator equivalent:
\\$\sin [\pi (1-\mathbf{B})/2]=\sin (\pi \mathbf{B}/2) = \pi /4$(Note that we have used the equation ${(1 - \mathbf{B})^n} = {\mathbf{B}^n}$,n is non-negative integer).
\par The zeros (poles , singularity) of a Bernoulli operator function without taking Bernoulli operator are called \textbf{primitive zeros (poles , singularity)}.For example,the primitive zeros of function $f(z)=\sin [\pi (\mathbf{B} + z)/2]$ are -$\mathbf{B}$+4k (k is an integer);
\par The zeros (poles , singularity) of a Bernoulli operator function by taking Bernoulli operator are called \textbf{Bernoulli operator zeros (poles , singularity)}.For example,the Bernoulli operator zero of function $f(z)=1 + \sin \pi \mathbf{B}z = 1 + \frac{{\pi z}}{2}$ is $\frac{{ - 2}}{\pi }$.When there is no confusion,we say a Bernoulli operator zero (pole,singularity) is a zero (pole,singularity).
 \par In the following paper we must to distinguish between primitive equivalent and Bernoulli operator equivalent. Otherwise,we will obtain many wrong equations.For example,we taking the logarithms on both sides of a primitive equivalent equation which remains equivalent; But we taking the logarithms on both sides of a Bernoulli operator equivalent equation,then we obtain a wrong equation(e.g.,$\sin (\pi \mathbf{B}/2) = \pi /4$;$\log\sin (\pi \mathbf{B}/2) \ne \log\pi /4$).Another examples ${\mathbf{B}^2} \cdot \mathbf{B} = {\mathbf{B}^3} = 0 \ne 1/6 \cdot 1/2.$
\par From the above discussion, we know the value of ${e^{ - \mathbf{B}}}$ is $\frac{1}{{e - 1}}$, we naturally want to know what is the value of a general Bernoulli operator function? Using the Euler-Maclaurin formula (see \cite{B} p.104):
\[\begin{array}{l}
 \sum\limits_{n = M}^N {f'(n)}  = \int\limits_M^N {f'(x)} dx + \frac{1}{2}\left[ {f'(M) + f'(N)} \right] + \frac{{{B_2}}}{2}f''(x)|_M^N + \frac{{{B_4}}}{4}f''''(x)|_M^N + ... \\
  + \frac{{{B_{2v}}}}{{(2v)!!}}{f^{(2v)}}(x)|_M^N + {R_{2v}}, \\
 \end{array}\]
where \[{R_{2v}} = \frac{{ - 1}}{{(2v)!}}\int\limits_M^N {\overline {{B_{2v}}(x)} } {f^{(2v + 1)}}(x)dx.\]
\\Let $M=1$ ,${N \to \infty }$ and ${v \to \infty }$,if ${f(N) \to 0}$,${f'(N) \to 0}$, and ${{f^{(2v)}}(N) \to 0}$ ,then
\[\sum\limits_{n = 1}^\infty  {f'(n)}  =  - f(1) + \frac{1}{2}f'(1) - \frac{{{B_2}}}{2}f''(1) - \frac{{{B_4}}}{4}f''''(1) + ... - \frac{{{B_{2v}}}}{{(2v)!!}}{f^{(2v)}}(1) + ...\]
\[ =  - f(1 - \mathbf{B}).\]
Since ${(1 - \mathbf{B})^n} = {\mathbf{B}^n}$ ,we obtain
\begin{equation}\label{1.1}
f(\mathbf{B}) =  - \sum\limits_{n = 1}^\infty  {f'(n)} .
\end{equation}
\par A Taylor expansion of Bernoulli operator function is ${f(\mathbf{B} + a) = \sum\limits_{n = 0}^\infty  {{f^{(n)}}(a)} \frac{{{{(\mathbf{B} - a)}^n}}}{{n!}}}$,\\if the right of the expansion does not converge in a certain areas of the complex plane, we use \eqref{1.1} that is ${f(\mathbf{B} + a) =  - \sum\limits_{n = 1}^\infty  {f'(n + a)} }$ to expand the domain of the function ${f(\mathbf{B} + a)}$; Conversely, if the right of the expansion ${f(\mathbf{B} + a) =  - \sum\limits_{n = 1}^\infty  {f'(n + a)}}$ does not converge in a certain areas of the complex plane, we use the expansion ${f(\mathbf{B} + a) = \sum\limits_{n = 0}^\infty  {{f^{(n)}}(a)} \frac{{{{(\mathbf{B} - a)}^n}}}{{n!}}}$  to expand the domain of the function ${f(\mathbf{B} + a)}$.This is a principle of analytic continuation.
\par For example,the following Bernoulli operator function is analytic at all points of the complex z-plane except for some poles at $z = 2\pi ki,$
\[{e^{ - \mathbf{B}z}} = \sum\limits_{n = 0}^\infty  {\frac{{{{( - 1)}^n}{\mathbf{B}^n}{z^n}}}{{n!}}}  = \sum\limits_{n = 1}^\infty  {z{e^{ - nz}}}  = \frac{z}{{{e^z} - 1}}.\]
\par From the above ideas and \eqref{1.1},we can obtain the formula,
\begin{equation}
{\mathbf{B}^{1 - s}} = \zeta (s)(s - 1).
\end{equation}
Therefore
 \[{\mathbf{B}^{ - 1}} = \zeta (2) = {\pi ^2}/6.\]
Similarly,we can obtain the following formulas
\begin{equation}
{(\mathbf{B} + n)^{1 - s}} = [\zeta (s) - {1^{ - s}} - {2^{ - s}} - ... - {n^{ - s}}](s - 1),
\end{equation}
\begin{equation}
{(\mathbf{B} + \alpha )^{1 - s}} = \zeta (s,\alpha )(s - 1),
\end{equation}
where ${\alpha  > 0,\zeta (s,\alpha )}$ is Hurwitz Zeta function.Moreover,using a formula of \cite{C} p.249 and above (1.4), we can obtain the following formula
\[L(s,\chi ) = {k^{ - s}}\sum\limits_{r = 1}^k {\chi (r)} \zeta (s,\frac{r}{k}) = {k^{ - s}}\sum\limits_{r = 1}^k {\chi (r)} \frac{{{{(\mathbf{B} + \frac{r}{k})}^{1 - s}}}}{{s - 1}}.\]
\par We taking the derivative of an equation of Bernoulli operator equavilent carries a new equation of Bernoulli operator equavilent(the Bernoulli operator $\mathbf{B}$ can be thought of as a constant).For example,
\[({e^{ - \mathbf{B}z}})' = (\frac{z}{{{e^z} - 1}})',\]
we obtain the equation \[-\mathbf{B}{e^{ - \mathbf{B}z}} = \frac{{{e^z} - 1 - z{e^z}}}{{{{({e^z} - 1)}^2}}}.\]
Similarly, we taking the derivatives on both sides of (1.2),we obtain
\begin{equation}
- {\mathbf{B}^{1 - s}}\log \mathbf{B} = \zeta (s) + (s - 1)\zeta '(s).
\end{equation}
And now,we have the following equation
\begin{equation}
\mathop {\lim }\limits_{s \to 1} [\zeta (s) + (s - 1)\zeta '(s)] = \gamma,
\end{equation}
where ${\gamma}$ is Euler constant,
therefore
\begin{equation}\label{1.7}
\gamma  =  - \log \mathbf{B}.
\end{equation}
Using \eqref{1.7} ,we give a proof of formula ${\gamma  = {\sum\limits_{n = 2}^\infty  {( - 1)} ^n}\frac{{\zeta (n)}}{n}}$(see \cite{D} p.5).
\begin{proof}
\par Since ${\mathbf{B}^{2n + 1}} = 0$(n is positive integer),applying the Taylor expansion,we readily find that
\[\mathbf{B}\log (1 + \mathbf{B}) =  - \mathbf{B}\log (1 - \mathbf{B}),\]
and we have
\[\mathbf{B}\log (1 + \frac{1}{\mathbf{B}}) = \mathbf{B}\log (1 + \mathbf{B}) - \mathbf{B}\log \mathbf{B},\]
therefore
\begin{equation}
\mathbf{B}\log (1 + \frac{1}{\mathbf{B}}) =  - \mathbf{B}\log (1 - \mathbf{B}) - \mathbf{B}\log \mathbf{B}.
\end{equation}
Since ${(1 - \mathbf{B})^n} = {\mathbf{B}^n}$,we have
\begin{equation}
 - \mathbf{B}\log (1 - \mathbf{B}) =  - (1 - \mathbf{B})\log \mathbf{B}.
\end{equation}
Combining (1.8)and(1.9),we obtain
\begin{equation}
\mathbf{B}\log (1 + \frac{1}{\mathbf{B}}) =  - \log \mathbf{B}.
\end{equation}
Combining \eqref{1.7} and(1.10),we obtain
\begin{equation}
\gamma  = \mathbf{B}\log (1 + \frac{1}{\mathbf{B}}).
\end{equation}
Now we have
\begin{equation}
 \log (1 + {\mathbf{B}^{ - 1}}) = \log (1 + \mathbf{B}) - \log \mathbf{B},
\end{equation}
and applying the Taylor expansion,we readily find that
\[\log (1 + \mathbf{B}) = 1 + \log (1 - \mathbf{B}) = 1 + \log \mathbf{B},\]
therefore
\begin{equation}
\log (1 + {\mathbf{B}^{ - 1}}) = 1 + \log \mathbf{B} - \log \mathbf{B} = 1.
\end{equation}
Combining(1.11)and(1.13),we obtain
\begin{equation}
\gamma  = \mathbf{B}\log (1 + \frac{1}{\mathbf{B}}) + \log (1 + {\mathbf{B}^{ - 1}}) - 1 = {\sum\limits_{n = 2}^\infty  {( - 1)} ^n}\frac{{{\mathbf{B}^{1 - n}}}}{{n(n - 1)}}.
\end{equation}
By (1.2),we show that
\begin{equation}
{\mathbf{B}^{1 - n}} = \zeta (n)(n - 1),
\end{equation}
Combining (1.14) and(1.15),we deduce that
\begin{equation}
\gamma  = {\sum\limits_{n = 2}^\infty  {( - 1)} ^n}\frac{{\zeta (n)}}{n}.
\end{equation}
\end{proof}
Similar to complex plane,we define a formal \textbf{Bernoulli operator plane} or simply Bernoulli plane,which is $\mathbb{B}:=\{ a + b\mathbf{B}i|a,b \in \mathbb{R} \}$.If we
define an "analytic function" on the Bernoulli plane,and define a contour integral for the analytic function,then we can deduce the classic Cauchy integral theorem,Cauchy integral formula and Residue theorem.
\par We can obtain another proof of (1.2) by Euler's integral for $\Pi (s - 1)$,where $\Pi (s - 1): = \int\limits_0^\infty  {{e^{ - x}}{x^{s - 1}}dx} $.Substitution of $nx$ for $x$ gives
\[\int\limits_0^\infty  {{e^{ - nx}}{x^{s - 1}}dx = \frac{{\Pi (s - 1)}}{{{n^s}}}} .\] We sum this over n to obtain
\[\int\limits_0^\infty  {\frac{{{x^{s - 1}}}}{{{e^x} - 1}}dx = \Pi (s - 1)\zeta (s)} .\]
Therefore \[\Pi (s - 1)\zeta (s) = \int\limits_0^\infty  {\frac{{{x^{s - 2}} \cdot x}}{{{e^x} - 1}}dx}  = \int\limits_0^\infty  {{x^{s - 2}}{e^{ - \mathbf{B}x}}dx} .\]
Let $\mathbf{B}x \to t$,note that $\mathbf{B}\infty$ can be thought of as $\infty $,we deduce that
\[\Pi (s - 1)\zeta (s) = {\mathbf{B}^{1 - s}}\int\limits_0^\infty  {{t^{s - 2}}{e^{ - t}}dt}  = {\mathbf{B}^{1 - s}}\Pi (s - 2) = {\mathbf{B}^{1 - s}}\frac{{\Pi (s - 1)}}{{s - 1}}.\]
Therefore we obtain (1.2) again.
\begin{remark}
I find that \eqref{1.1} is similar to Ramanujan summation,and he has researched divergent series by his summation(see \cite{F}).On the other hand,C.Vignat told me that the Bernoulli operator somewhat similar to Bernoulli umbra (see \cite{G}).
And J.G¨¦linas told me that \eqref{1.7} dates from 1861 in the original articles of Blissard (see \cite{J}).

\end{remark}
\begin{remark}
If a point of the complex z-plane is Bernoulli operator pole( or singularity) of a function,then using the transformation  $\mathbf{B} \to 1 - \mathbf{B}$ will cause errors.For example,
\[{e^{ - 2\pi i\mathbf{B}}} = \cos 2\pi \mathbf{B} - i\sin 2\pi \mathbf{B},\] if let $\mathbf{B} \to 1 - \mathbf{B}$,then
\[{e^{ - 2\pi i(1 - \mathbf{B})}} = {e^{ - 2\pi i + 2\pi i\mathbf{B}}} = {e^{2\pi i\mathbf{B}}}\]
\[ = \cos 2\pi \mathbf{B} + i\sin 2\pi \mathbf{B}.\]
Therefore we obtain an error equation
\[ - i\sin 2\pi \mathbf{B} = i\sin 2\pi \mathbf{B}.\]
\end{remark}


\vspace{0.3in}
\section{\textbf{Distribution of Bernoulli operator zeros of the function ${\xi (\mathbf{B} + s)}$ and ${\sin \pi \mathbf{B} \cdot \xi (\mathbf{B} + s)}$ }}
\begin{theorem}
The function ${\xi (\mathbf{B} + s)}$ and ${\sin \pi \mathbf{B} \cdot \xi (\mathbf{B} + s)}$ satisfy the following functional equations respectively
\begin{equation}
\xi (\mathbf{B} + s) = \xi (\mathbf{B} - s),
\end{equation}
\begin{equation}
\sin \pi \mathbf{B} \cdot \xi (\mathbf{B} + s) = \sin \pi \mathbf{B} \cdot \xi (\mathbf{B} - s).
\end{equation}
The values of these two functions  are positive in the real axis;and these two functions have infinitely many Bernoulli operator zeros on the imaginary axis.
\end{theorem}
\begin{proof}
By ${\xi (s) = \xi (1 - s)}$,and let ${s \to \mathbf{B} + s}$,we have
\begin{equation}
\xi (\mathbf{B} + s) = \xi (1 - \mathbf{B} - s).
\end{equation}
Since ${(1 - \mathbf{B})^n} = {\mathbf{B}^n}$,we have
\begin{equation}
\xi (1 - \mathbf{B} - s) = \xi (\mathbf{B} - s).
\end{equation}
Using(2.3)and(2.4),we obtain(2.1).
Similarly,we have
\[\sin \pi \mathbf{B} \cdot \xi (\mathbf{B} + s) = \sin \pi (1 - \mathbf{B}) \cdot \xi (1 - \mathbf{B} + s) = \sin \pi (1 - \mathbf{B}) \cdot \xi (\mathbf{B} - s).\]
Since $\sin \pi (1 - \mathbf{B})$ is  primitive equal to $\sin \pi \mathbf{B}$,we have
\[\sin \pi (1 - \mathbf{B}) \cdot \xi (\mathbf{B} - s) = \sin \pi \mathbf{B} \cdot \xi (\mathbf{B} - s).\]
Therefore
\[\sin \pi \mathbf{B} \cdot \xi (\mathbf{B} + s) = \sin \pi \mathbf{B} \cdot \xi (\mathbf{B} - s).\]
\par We now prove that the function ${\xi (\mathbf{B} + s)}$ is positive in the real axis. We have the following equation (see \cite{B}
p.17), \[\xi (s) = \int\limits_1^\infty  {\frac{{d[{x^{3/2}}\psi '(x)]}}{{dx}}} (2{x^{\frac{{s - 1}}{2}}} + 2{x^{\frac{{ - s}}{2}}})dx,\]where $\psi '(x) = \frac{{d[\sum\limits_{n = 1}^\infty  {{e^{ - {n^2}\pi x}}} ]}}{{dx}},$
therefore
\[\xi (\mathbf{B} + s) = \int\limits_1^\infty  {\frac{{d[{x^{3/2}}\psi '(x)]}}{{dx}}} (2{x^{\frac{{s - \mathbf{B}}}{2}}} + 2{x^{\frac{{ - s - \mathbf{B}}}{2}}})dx.\]
Since ${e^{ - \mathbf{B}x}} = \frac{x}{{{e^x} - 1}}$, we have ${x^{ - \frac{\mathbf{B}}{2}}} = \frac{{\log x}}{{2({x^{1/2}} - 1)}}$. Let $\phi (x) = \frac{{d[{x^{3/2}}\psi '(x)]}}{{dx}}\frac{{\log x}}{{{x^{1/2}} - 1}}$,then
\[\xi (\mathbf{B} + s) = \int\limits_1^\infty  {\phi (x)} ({x^{\frac{s}{2}}} + {x^{\frac{{ - s}}{2}}})dx.\]
Because $\phi (x) > 0(x \in (1,\infty )),$ we obtain
$\xi (\mathbf{B} + s) > 0,s \in ( - \infty ,\infty ).$
Using \eqref{1.1},we conclude that $\sin \pi \mathbf{B} \cdot {x^{ - \frac{\mathbf{B}}{2}}} = \frac{\pi }{{{x^{\frac{1}{2}}} + 1}}$ .Similarly,we can prove that $\sin \pi \mathbf{B} \cdot \xi (\mathbf{B} + s) > 0,s \in ( - \infty ,\infty ).$
\par We now prove that the function $\xi (\mathbf{B} + s)$ has infinitely many Bernoulli operator zeros on the imaginary axis.
\\Let $G(x) = \sum\limits_{n =  - \infty }^\infty  {{e^{ - \pi {n^2}{x^2}}}} $,$H(x) = x \frac{{{d^2}[xG(x) - x - 1]}}{{{d}x^2}}$,
then we have the following equation (see \cite{B} p. 228),
\[H(x) = \frac{1}{\pi }\int\limits_{ - \infty }^\infty  {\xi (a + it){x^{a - 1}}} {x^{it}}dt,\]
multiply both sides of above equation by ${x^{1 - a}}$
\[{x^{1 - a}}H(x) = \frac{1}{\pi }\int\limits_{ - \infty }^\infty  {\xi (a + it)} {x^{it}}dt.\]
Let  $a \to \mathbf{B}$($\mathbf{B}$ is Bernoulli operator),then
\[{x^{1 - \mathbf{B}}}H(x) = \frac{1}{\pi }\int\limits_{ - \infty }^\infty  {\xi (\mathbf{B} + it)} {x^{it}}dt,\]
\[ = \frac{1}{\pi }\int\limits_{ - \infty }^\infty  {\xi (\mathbf{B} + it)} \sum\limits_{n = 0}^\infty  {\frac{{{{(it\log x)}^n}}}{{n!}}} dt.\]
Denote ${C_n} = \frac{1}{{\pi n!}}\int\limits_{ - \infty }^\infty  {\xi (\mathbf{B} + it)} {t^n}dt$,then
\[{x^{1 - \mathbf{B}}}H(x) = \frac{{x\log x}}{{x - 1}}H(x) = {\sum\limits_{n = 0}^\infty  {{C_n}(i\log x)} ^n}.\]
Since $\xi (\mathbf{B} + it) = \xi (\mathbf{B} - it)$,we have ${C_n} = 0$ (n is odd).
The differential operator $ix(d/dx)$ doing $\frac{{x\log x}}{{x - 1}}H(x)$ any number of  times carries $\frac{{x\log x}}{{x - 1}}H(x)$ to a function which approaches zero as $x \to {i^{\frac{1}{2}}}$.Therefore, according to \cite{B} p.228-229 ,we can prove that function $\xi (\mathbf{B} + s)$  has infinitely many Bernoulli operator zeros on the imaginary axis.
\par Similarly,we can prove that function $\sin \pi \mathbf{B} \cdot \xi (\mathbf{B} + s)$ has infinitely many Bernoulli operator zeros on the imaginary axis.
\end{proof}
\begin{conjecture}
All zeros of the function $\xi (\mathbf{B} + s)$ ($\sin \pi \mathbf{B} \cdot \xi (\mathbf{B} + s)$) lie on the imaginary axis.
\end{conjecture}
\par The Riemann Hypothesis generalized to Dirichlet L-functions,Dedekind Zeta function (see \cite{C} p.176)and Zeta Functions of varieties over Finite Fields,etc.Similarly,the conjecture 2.2 can be generalized to other Zeta functions(L-functions).

\vspace{0.3in}
\section{\textbf{An application of Bernoulli operator to number-theoretic function}}
\begin{theorem}
we define a number-theoretic function:
\[\psi (x) = \frac{1}{2}\left[ {\sum\limits_{{p^n} < x} {\frac{{\log p}}{{{p^n} - 1}} + \sum\limits_{{p^n} \le x} {\frac{{\log p}}{{{p^n} - 1}}} } } \right],\]
we have
\[\begin{array}{l}
 \psi (x) = \log (x - 1) - \mathop {\lim }\limits_{\varepsilon  \to 0} \sum\limits_{{\mathop{\rm Im}\nolimits} \rho  > 0} {(\int\limits_0^{1 - \varepsilon } {\frac{{{t^{\rho  - 1}} + {t^{ - \rho }}}}{{t - 1}}} dt + \int\limits_{1 + \varepsilon }^x {\frac{{{t^{\rho  - 1}} + {t^{ - \rho }}}}{{t - 1}}dt} )}  \\
  + \int\limits_x^\infty  {\frac{{dt}}{{t(t - 1)({t^2} - 1)}}}  + \log [ - \zeta (\mathbf{B})], \\
 \end{array}\]
where $\rho $ are zeros of $\xi (s)$,and $x >1$.
\end{theorem}
\begin{proof}
By Euler product formula
\[\log \zeta (s) = \sum\limits_p {[\sum\limits_{n = 1}^\infty  {\frac{1}{n}} {p^{ - ns}}]} ,({\rm{Re}}{\kern 1pt} s > 1),\]
let $s \to s + \mathbf{B}$,therefore
 \[\log \zeta (s + \mathbf{B}) = \sum\limits_p {[\sum\limits_{n = 1}^\infty  {\frac{1}{n}} {p^{ - ns - n\mathbf{B}}}]}  = \sum\limits_p {[\sum\limits_{n = 1}^\infty  {\frac{1}{n}} {p^{ - ns}}\frac{{\log p}}{{{p^n} - 1}}]} \]
 \[  < \sum\limits_p {[\sum\limits_{n = 1}^\infty  {\frac{1}{n}p \cdot } {p^{ - n(s + 1)}}]} .\]
The above series on the right is absolutely convergent for ${\mathop{\rm Re}\nolimits} {\kern 1pt} s > 0 $.We write this sum as a Stieltjes integral
 \[\log \zeta (s + \mathbf{B}) = \int\limits_0^\infty  {{x^{ - s}}} d\psi (x) = s\int\limits_0^\infty  {\psi (x)} {x^{ - s - 1}}dx. ({\mathop{\rm Re}\nolimits} {\kern 1pt} s > 0)\]
 We applies Fourier inversion to above formula,to conclude
 \[\psi (x) = \frac{1}{{2\pi i}}\int\limits_{a - i\infty }^{a + i\infty } {{x^s}\log \zeta (s + \mathbf{B})} \frac{{ds}}{s}. (a > 0)\]
 We integrates by parts to obtain
 \begin{equation}
 \psi (x) = \frac{{ - 1}}{{2\pi i}}\frac{1}{{\log x}}\int\limits_{a - i\infty }^{a + i\infty } {\frac{d}{{ds}}[\frac{{\log \zeta (s + \mathbf{B})}}{s}} ]{x^s}ds.
 \end{equation}
On the other hand,we have
 \[\xi (\mathbf{B} + s) = \Pi (\frac{{\mathbf{B} + s}}{2}){\pi ^{ - \frac{{s + \mathbf{B}}}{2}}}(s + \mathbf{B} - 1)\zeta (s + \mathbf{B}),\]
 where $\Pi (s): = \int\limits_0^\infty  {{e^{ - x}}{x^s}dx} $,
 \\and
 \\ $\xi (\mathbf{B} + s) = \xi (\mathbf{B})\prod\limits_\rho  {(1 - \frac{s}{{\rho  - \mathbf{B}}})} $,($\rho $ are zeros of $\xi (s)$,and $\rho -\mathbf{B} $ are primitive zeros of $\xi (s+\mathbf{B})$ ).
 \\Therefore
 \begin{equation}
 \begin{array}{l}
 \log \zeta (s + \mathbf{B}) =  - \log \Pi (\frac{{\mathbf{B} + s}}{2}) + \frac{{s + \mathbf{B}}}{2}\log \pi  - \log (s + \mathbf{B} - 1) \\
  + \log \xi (\mathbf{B}) + \sum\limits_\rho  {\log (1 - \frac{s}{{\rho  - \mathbf{B}}})} . \\
 \end{array}
 \end{equation}
 Combining(3.1)and(3.2),we obtain
 \begin{equation}
\psi (x) = \frac{{ - 1}}{{2\pi i}}\frac{1}{{\log x}}\int\limits_{a - i\infty }^{a + i\infty } {\frac{d}{{ds}}[\frac{{ - \log \Pi (\frac{{\mathbf{B} + s}}{2}) + \frac{{s + \mathbf{B}}}{2}\log \pi  - \log (s + \mathbf{B} - 1) + \log \xi (\mathbf{B}) + \sum\limits_\rho  {\log (1 - \frac{s}{{\rho  - \mathbf{B}}})} }}{s}} ]{x^s}ds.
 \end{equation}
 According to \cite{B} p.26-31,we show that
 \[\frac{1}{{2\pi i}}\frac{1}{{\log x}}\int\limits_{a - i\infty }^{a + i\infty } {\frac{d}{{ds}}[\frac{{\log (s + \mathbf{B} - 1)}}{s}]} {x^s}ds = \frac{1}{{2\pi i}}\frac{1}{{\log x}}\int\limits_{a - i\infty }^{a + i\infty } {\frac{d}{{ds}}[\frac{{\log (s/(\mathbf{B} - 1) + 1) + \log (\mathbf{B} - 1)}}{s}]} {x^s}ds\]
 \[= \frac{1}{{2\pi i}}\frac{1}{{\log x}}\int\limits_{a - i\infty }^{a + i\infty } {\frac{d}{{ds}}[\frac{{\log (1 - s/\mathbf{B}) + \log (\mathbf{B} - 1)}}{s}]} {x^s}ds\]
 \[= \frac{1}{{2\pi i}}\frac{1}{{\log x}}\int\limits_{a - i\infty }^{a + i\infty } {\frac{d}{{ds}}[\frac{{\log (s/\mathbf{B} - 1) + \log (1 - \mathbf{B})}}{s}]} {x^s}ds\]
 \[ = \mathop {\lim }\limits_{\varepsilon  \to 0} (\int\limits_0^{1 - \varepsilon } {\frac{{{t^{\mathbf{B} - 1}}}}{{\log t}}} dt + \int\limits_{1 + \varepsilon }^x {\frac{{{t^{\mathbf{B} - 1}}}}{{\log t}}} dt) - \log (1 - \mathbf{B})\]
 \[ = \mathop {\lim }\limits_{\varepsilon  \to 0} (\int\limits_0^{1 - \varepsilon } {\frac{{{t^{ - \mathbf{B}}}}}{{\log t}}} dt + \int\limits_{1 + \varepsilon }^x {\frac{{{t^{ - \mathbf{B}}}}}{{\log t}}} dt) - \log (1 - \mathbf{B})\]
 \begin{equation}
 = \mathop {\lim }\limits_{\varepsilon  \to 0} (\int\limits_0^{1 - \varepsilon } {\frac{1}{{t - 1}}} dt + \int\limits_{1 + \varepsilon }^x {\frac{1}{{t - 1}}} dt) - \log (1 - \mathbf{B}).
 \end{equation}
 If $x >1 $,then
 \[\mathop {\lim }\limits_{\varepsilon  \to 0} (\int\limits_0^{1 - \varepsilon } {\frac{1}{{t - 1}}} dt + \int\limits_{1 + \varepsilon }^x {\frac{1}{{t - 1}}} dt) = \log (x - 1).\]
 Therefore
 \begin{equation}
\frac{1}{{2\pi i}}\frac{1}{{\log x}}\int\limits_{a - i\infty }^{a + i\infty } {\frac{d}{{ds}}[\frac{{\log (s + \mathbf{B} - 1)}}{s}]} {x^s}ds = \log (x - 1) - \log (1 - \mathbf{B}).
 \end{equation}
 According to \cite{B} p.26-31,we show that
 \[\frac{1}{{2\pi i}}\frac{1}{{\log x}}\int\limits_{a - i\infty }^{a + i\infty } {\frac{d}{{ds}}[\frac{{\sum\limits_\rho  {\log (1 - \frac{s}{{\rho  - \mathbf{B}}})} }}{s}} ]{x^s}ds = \sum\limits_{{\mathop{\rm Im}\nolimits} \rho  > 0} {[Li({x^{\rho  - \mathbf{B}}}) + Li({x^{1 - \rho  - \mathbf{B}}})} ]\]
 \[ = \sum\limits_{{\mathop{\rm Im}\nolimits} \rho  > 0} {[Li({x^{\rho  - \mathbf{B}}}) + Li({x^{\mathbf{B} - \rho }})} ]\]
 \begin{equation}
 = \mathop {\lim }\limits_{\varepsilon  \to 0} \sum\limits_{{\mathop{\rm Im}\nolimits} \rho  > 0} {(\int\limits_0^{1 - \varepsilon } {\frac{{{t^{\rho  - 1}} + {t^{ - \rho }}}}{{t - 1}}} dt + \int\limits_{1 + \varepsilon }^x {\frac{{{t^{\rho  - 1}} + {t^{ - \rho }}}}{{t - 1}}dt} )}.
 \end{equation}
 Using a formula form \cite{B} p.8,we show that
 \[\log \Pi (\frac{{s + \mathbf{B}}}{2}) = \sum\limits_{n = 1}^\infty  {\left[ { - \log (1 + \frac{{s + \mathbf{B}}}{{2n}}) + \frac{{s + \mathbf{B}}}{2}\log (1 + \frac{1}{n})} \right]}. \]
 Since
 \[\log (1 + \frac{{s + \mathbf{B}}}{{2n}}) = \log (1 + \frac{\mathbf{B}}{{2n}} + \frac{s}{{2n}}) = \log (1 + \frac{\mathbf{B}}{{2n}}) + \log (1 + \frac{s}{{\mathbf{B} + 2n}}),\]
 we obtain
 \[\frac{1}{{2\pi i}}\frac{1}{{\log x}}\int\limits_{a - i\infty }^{a + i\infty } {\frac{d}{{ds}}[\frac{{\log \Pi (\frac{{s + \mathbf{B}}}{2})}}{s}} ]{x^s}ds\]
 \[ = \frac{1}{{2\pi i}}\frac{1}{{\log x}}\sum\limits_{n = 1}^\infty  {\int\limits_{a - i\infty }^{a + i\infty } {\frac{d}{{ds}}\left[ { - \log (1 + \frac{\mathbf{B}}{{2n}}) - \log (1 + \frac{s}{{\mathbf{B} + 2n}}) + \frac{{s + \mathbf{B}}}{2}\log (1 + \frac{1}{n})} \right]{x^s}} } ds\]
 \[ = \frac{1}{{2\pi i}}\frac{1}{{\log x}}\sum\limits_{n = 1}^\infty  {\int\limits_{a - i\infty }^{a + i\infty } {\frac{d}{{ds}}\left[ { - \log (1 + \frac{\mathbf{B}}{{2n}}) - \log (1 + \frac{s}{{\mathbf{B} + 2n}}) + \frac{\mathbf{B}}{2}\log (1 + \frac{1}{n})} \right]{x^s}} } ds\]
 \[ = \sum\limits_{n = 1}^\infty  {[\log (1 + \frac{\mathbf{B}}{{2n}})}  - \frac{\mathbf{B}}{2}\log (1 + \frac{1}{n})] - \frac{1}{{2\pi i}}\frac{1}{{\log x}}\sum\limits_{n = 1}^\infty  {\int\limits_{a - i\infty }^{a + i\infty } {\frac{d}{{ds}}\left[ {\log (1 + \frac{s}{{\mathbf{B} + 2n}})} \right]{x^s}} } ds\]
 \[ =  - \log \Pi (\frac{\mathbf{B}}{2}) - \frac{1}{{2\pi i}}\frac{1}{{\log x}}\sum\limits_{n = 1}^\infty  {\int\limits_{a - i\infty }^{a + i\infty } {\frac{d}{{ds}}\left[ {\log (1 + \frac{s}{{\mathbf{B} + 2n}})} \right]{x^s}} } ds\]
 \begin{equation}
  =  - \log \Pi (\frac{\mathbf{B}}{2}) - \sum\limits_{n = 1}^\infty  {H( - 2n - \mathbf{B})}.
 \end{equation}
 Where  the function  $H(x)$ defined in \cite{B} p.28.According to \cite{B} p.32,we show that
 \[\sum\limits_{n = 1}^\infty  {H( - 2n - \mathbf{B})}  = \int\limits_x^\infty  {\frac{{dt}}{{t(t - 1)({t^2} - 1)}}} .\]
 Therefore
 \begin{equation}
\frac{1}{{2\pi i}}\frac{1}{{\log x}}\int\limits_{a - i\infty }^{a + i\infty } {\frac{d}{{ds}}[\frac{{\log \Pi (\frac{{s + \mathbf{B}}}{2})}}{s}} ]{x^s}ds =  - \log \Pi (\frac{\mathbf{B}}{2}) - \int\limits_x^\infty  {\frac{{dt}}{{t(t - 1)({t^2} - 1)}}}.
 \end{equation}
 Using(3.3),(3.5),(3.6)and(3.8),we obtain
 \begin{equation}
\begin{array}{l}
 \psi (x) = \log (x - 1) - \log (1 - \mathbf{B}) + \frac{\mathbf{B}}{2}\log \pi  + \log \xi (\mathbf{B}) \\
  - \mathop {\lim }\limits_{\varepsilon  \to 0} \sum\limits_{{\mathop{\rm Im}\nolimits} \rho  > 0} {(\int\limits_0^{1 - \varepsilon } {\frac{{{t^{\rho  - 1}} + {t^{ - \rho }}}}{{t - 1}}} dt + \int\limits_{1 + \varepsilon }^x {\frac{{{t^{\rho  - 1}} + {t^{ - \rho }}}}{{t - 1}}dt} )}  - \log \Pi (\frac{\mathbf{B}}{2}) + \int\limits_x^\infty  {\frac{{dt}}{{t(t - 1)({t^2} - 1)}}} . \\
 \end{array}
 \end{equation}
On the other hand,we have
 \[\log \xi (\mathbf{B}) = \log \Pi (\frac{\mathbf{B}}{2}) + \log (1 - \mathbf{B}) - \frac{\mathbf{B}}{2}\log \pi  + \log [ - \zeta (\mathbf{B})].\]
 By the above equation and (3.9),we complete the proof.
\end{proof}
\begin{remark}
\[\sin \pi \mathbf{B} \cdot \log \zeta (\mathbf{B} + s) = \sin \pi \mathbf{B} \cdot \sum\limits_p {[\sum\limits_{n = 1}^\infty  {\frac{1}{n}} {p^{ - n(s + \mathbf{B})}}]} ,\]
Using \eqref{1.1},we conclude that $\sin \pi \mathbf{B}\cdot{p^{ - n\mathbf{B}}} = \frac{\pi }{{{p^n} + 1}}$,therefore
\[\sin \pi \mathbf{B}\cdot\log \zeta (\mathbf{B} + s) = \pi \sum\limits_p {[\sum\limits_{n = 1}^\infty  {\frac{1}{n}} {p^{ - ns}}\frac{1}{{{p^n} + 1}}]} .\]
The above series on the right is absolutely convergent for ${\mathop{\rm Re}\nolimits} {\kern 1pt} s > 0$.
\end{remark}

\vspace{0.3in}
\section{\textbf{The values of $\mathbf{B}\log \mathbf{B},\log \sin \frac{{\pi \mathbf{B}}}{2},\log \Pi (\mathbf{B}),\frac{{\zeta '(\mathbf{B})}}{{\zeta (\mathbf{B})}}$ and $\sin \pi \mathbf{B}\cdot \frac{{\zeta '(\mathbf{B})}}{{\zeta (\mathbf{B})}}$}}
\begin{Proposition}
The values of $\mathbf{B}\log \mathbf{B},\log \sin \frac{{\pi \mathbf{B}}}{2},\log \Pi (\mathbf{B}),\frac{{\zeta '(\mathbf{B})}}{{\zeta (\mathbf{B})}}$ and $\sin \pi \mathbf{B}\cdot\frac{{\zeta '(\mathbf{B})}}{{\zeta (\mathbf{B})}}$ are
\[\frac{{1 - \log 2\pi }}{2},\frac{1}{2} - \log 2,\frac{{\log 2\pi  - 1}}{2} - \gamma ,\frac{{1 + \gamma  + \log 2\pi }}{2} + \frac{{{\pi ^2}}}{{16}}\] and $\frac{\pi }{4}(1 + \gamma  + \log 4\pi )$ respectively.
\end{Proposition}
\begin{lemma}
If a function $f(z)$ has the following Taylor expansion in the neighborhood of $z = a$
\[\sum\limits_{n = 0}^\infty  {\frac{{{f^{(n)}}(a)}}{{n!}}} {(z - a)^n},\]
then
\begin{equation}
f(a + \mathbf{B}z) = f(a - \mathbf{B}z) + f'(a)z,
\end{equation}
\begin{equation}
\mathbf{B}f'(a + \mathbf{B}z) =  - \mathbf{B}f'(a - \mathbf{B}z) + f'(a),
\end{equation}
\begin{equation}
{\mathbf{B}^2}f''(a + \mathbf{B}z) = {\mathbf{B}^2}f''(a - \mathbf{B}z).
\end{equation}
\end{lemma}
\begin{proof}
It is clear from the properties of Bernoulli numbers.
\begin{remark}
We can obtain \eqref{1.1} by (4.1).Using (4.1),we have
\[f(1 - \mathbf{B}) = f(1 + \mathbf{B}) - f'(1),\]
Since $f(\mathbf{B}) = f(1 - \mathbf{B})$,we have
\[f(\mathbf{B}) = f(1 + \mathbf{B}) - f'(1).\]
Therefore
\[f(\mathbf{B}) =f(1+1- \mathbf{B}) - f'(1)=f(2- \mathbf{B}) - f'(1).\]
using (4.1) again,
\[f(\mathbf{B}) = f(2+ \mathbf{B}) - f'(1)-f'(2).\]
Therefore
\[f(\mathbf{B}) = f(n + \mathbf{B}) - \sum\limits_{k = 1}^n {f'(k)}.\]
Let $n \to \infty $,then we obtain \eqref{1.1} again.
\end{remark}

\par Now we prove that Proposition 4.1.Set $s = 0$ in (1.5),we obtain
\begin{equation}
- \mathbf{B}\log \mathbf{B} = \zeta (0) - \zeta '(0) =  - \frac{1}{2} - \zeta '(0).
\end{equation}
And we have the following equation (see \cite{B} p.135),
\begin{equation}
\zeta '(0) =  - \frac{1}{2}\log 2\pi.
\end{equation}
Therefore
\begin{equation}
\mathbf{B}\log \mathbf{B} = \frac{{1 - \log 2\pi }}{2}.
\end{equation}
\par Here we will use another way to prove (4.6).
We take the logarithms of both sides of the Riemann' Zeta functional equation,let $s \to \mathbf{B}$,we obtain
\[\log [ - \zeta (\mathbf{B})] + \log (1 - \mathbf{B}) = \log \Pi (1 - \mathbf{B}) + (\mathbf{B} - 1)\log 2\pi  + \log 2 + \log \sin \frac{{\pi \mathbf{B}}}{2} + \log [ - \zeta (1 - \mathbf{B})]\]
Therefore
\begin{equation}
 - \gamma  = \log \Pi (\mathbf{B}) - \frac{{\log 2\pi }}{2} + \log 2 + \log \sin \frac{{\pi \mathbf{B}}}{2}.
\end{equation}
Using the Stirling formula (see \cite{B} p.109),
\begin{equation}
\log \Pi (\mathbf{B}) = (\mathbf{B} + \frac{1}{2})\log \mathbf{B} - \mathbf{B} + \frac{{\log 2\pi }}{2} + \sum\limits_{n = 1}^\infty  {\frac{{{{( - 1)}^{n + 1}}B_{2n}^2{{(2\pi )}^{2n}}}}{{2n \cdot 2 \cdot (2n)!}}}.
\end{equation}
Note that we have used the following formula
\[{\mathbf{B}^{1 - 2k}} = \zeta (2k)(2k - 1) = \frac{{{{( - 1)}^{k + 1}}{{(2\pi )}^{2n}}{B_{2k}}}}{{2 \cdot (2k)!}},\]
where k is positive integer.
Since
\[\frac{{\cos z}}{{\sin z}} = \frac{1}{z} - {\sum\limits_{n = 1}^\infty {( - 1)}^{n + 1} {((2n)!)} ^{ - 1}}{2^{2n}}{B_{2n}}{z^{2n - 1}},\]
taking the integral on both sides to get
\begin{equation}
\log \sin z = \log z - \sum\limits_{n = 1}^\infty  {\frac{{{{( - 1)}^{n + 1}}{2^{2n}}{B_{2n}}}}{{2n \cdot (2n)!}}{z^{2n}}}.
\end{equation}(see \cite{I} p.39).
Let $z \to \mathbf{B}\pi $,we have
\begin{equation}
\log \sin \mathbf{B}\pi  = \log \mathbf{B} + \log \pi  - \sum\limits_{n = 1}^\infty  {\frac{{{{( - 1)}^{n + 1}}B_{2n}^2{{(2\pi )}^{2n}}}}{{2n \cdot (2n)!}}}.
 \end{equation}
Using(4.8)and(4.10),we obtain
\begin{equation}
\log \Pi (\mathbf{B}) = (\mathbf{B} + \frac{1}{2})\log \mathbf{B} - \mathbf{B} + \frac{{\log 2\pi }}{2} + \frac{{\log \mathbf{B} + \log \pi  - \log \sin \mathbf{B}\pi }}{2}.
\end{equation}
Using(4.7)and(4.11),we obtain
\[ - \gamma  = (\mathbf{B} + \frac{1}{2})\log \mathbf{B} - \mathbf{B} + \frac{{\log 2\pi }}{2} + \frac{{\log \mathbf{B} + \log \pi  - \log \sin \mathbf{B}\pi }}{2} - \frac{{\log 2\pi }}{2} + \log 2 + \log \sin \frac{{\pi \mathbf{B}}}{2}.\]
Therefore
\[\mathbf{B}\log \mathbf{B} = \frac{1}{2} - \log 2 - \frac{{\log \pi }}{2} + \frac{{\log \sin \mathbf{B}\pi  - 2\log \sin \frac{{\mathbf{B}\pi }}{2}}}{2}\]
\[ = \frac{1}{2} - \log 2 - \frac{{\log \pi }}{2} + \frac{{\log (2\sin \frac{{\mathbf{B}\pi }}{2}\cos \frac{{\mathbf{B}\pi }}{2}) - 2\log \sin \frac{{\mathbf{B}\pi }}{2}}}{2}.\]
\begin{equation}
 = \frac{1}{2} - \log 2 - \frac{{\log \pi }}{2} + \frac{{\log 2 + \log \cos \frac{{\mathbf{B}\pi }}{2} - \log \sin \frac{{\mathbf{B}\pi }}{2}}}{2}.
\end{equation}
And we have
\begin{equation}
\log \cos \frac{{\mathbf{B}\pi }}{2} = \log \cos \frac{{(1 - \mathbf{B})\pi }}{2} = \log \sin \frac{{\mathbf{B}\pi }}{2}.
\end{equation}
Combining(4.12) and(4.13),we obtain(4.6).
\par We began to calculate the value of $\log \sin \frac{{\pi \mathbf{B}}}{2}$ ,we have
\[\mathbf{B}\log \sin \frac{{\pi \mathbf{B}}}{2} = (1 - \mathbf{B})\log \sin [\frac{\pi }{2}(1 - \mathbf{B})] = (1 - \mathbf{B})\log \cos \frac{{\pi \mathbf{B}}}{2}\]
\[ = \log \cos \frac{{\pi \mathbf{B}}}{2} - \mathbf{B}\log \cos \frac{{\pi \mathbf{B}}}{2}\]
\begin{equation}
 = \log \sin \frac{{\pi \mathbf{B}}}{2} - \mathbf{B}\log \cos \frac{{\pi \mathbf{B}}}{2}.
\end{equation}
Using the Taylor expansion,we have
\begin{equation}
\mathbf{B}\log \cos \frac{{\pi \mathbf{B}}}{2} = 0.
\end{equation}
Using(4.9),we have
\[\mathbf{B}\log \sin \frac{{\pi \mathbf{B}}}{2} = \mathbf{B}\log \frac{{\pi \mathbf{B}}}{2} = \mathbf{B}\log \pi \mathbf{B} - \mathbf{B}\log 2 = \mathbf{B}\log \pi  + \mathbf{B}\log \mathbf{B} - \frac{{\log 2}}{2}\]
\begin{equation}
 = \frac{{\log \pi }}{2} + \frac{{1 - \log 2\pi }}{2} - \frac{{\log 2}}{2} = \frac{1}{2} - \log 2.
\end{equation}
Using(4.14),(4.15)and(4.16),we obtain
\begin{equation}
\log \sin \frac{{\pi \mathbf{B}}}{2} = \frac{1}{2} - \log 2.
\end{equation}
\par Using(4.17)and(4.7),we obtain
\[ - \gamma  = \log \Pi (\mathbf{B}) - \frac{{\log 2\pi }}{2} + \log 2 + \frac{1}{2} - \log 2.\]
Therefore
\begin{equation}
\log \Pi (\mathbf{B}) = \frac{{\log 2\pi  - 1}}{2} - \gamma.
\end{equation}

\par We calculate the values of $\frac{{\zeta '(\mathbf{B})}}{{\zeta (\mathbf{B})}}$  and $\sin \pi \mathbf{B}\cdot\frac{{\zeta '(\mathbf{B})}}{{\zeta (\mathbf{B})}}$  by the following equation,
\begin{equation}
\frac{{\zeta '(s)}}{{\zeta (s)}} =  - \frac{{\Pi '( - s)}}{{\Pi ( - s)}} + \log 2\pi  + \frac{\pi }{2}\frac{{\cos (\pi s/2)}}{{\sin (\pi s/2)}} - \frac{{\zeta '(1 - s)}}{{\zeta (1 - s)}}.
\end{equation}
Let $s \to \mathbf{B}$,both sides multiplied by $\mathbf{B}$,we obtain
\[\mathbf{B}\frac{{\zeta '(\mathbf{B})}}{{\zeta (\mathbf{B})}} =  - \mathbf{B}\frac{{\Pi '( - \mathbf{B})}}{{\Pi ( - \mathbf{B})}} + \mathbf{B}\log 2\pi  + \frac{\pi }{2}\mathbf{B}\frac{{\cos (\pi \mathbf{B}/2)}}{{\sin (\pi \mathbf{B}/2)}} - \mathbf{B}\frac{{\zeta '(1 - \mathbf{B})}}{{\zeta (1 - \mathbf{B})}}\]
\[ =  - \mathbf{B}\frac{{\Pi '( - \mathbf{B})}}{{\Pi ( - \mathbf{B})}} + \mathbf{B}\log 2\pi  + \frac{\pi }{2}\mathbf{B}\frac{{\cos (\pi \mathbf{B}/2)}}{{\sin (\pi \mathbf{B}/2)}} - (1 - \mathbf{B})\frac{{\zeta '(\mathbf{B})}}{{\zeta (\mathbf{B})}}.\]
Hence
\begin{equation}
\frac{{\zeta '(\mathbf{B})}}{{\zeta (\mathbf{B})}} =  - \mathbf{B}\frac{{\Pi '( - \mathbf{B})}}{{\Pi ( - \mathbf{B})}} + \mathbf{B}\log 2\pi  + \frac{\pi }{2}\mathbf{B}\frac{{\cos (\pi \mathbf{B}/2)}}{{\sin (\pi \mathbf{B}/2)}}.
\end{equation}
Firstly,we calculate the value of $\frac{\pi }{2}\mathbf{B}\frac{{\cos (\pi \mathbf{B}/2)}}{{\sin (\pi \mathbf{B}/2)}}$.We have
\begin{equation}
{\mathbf{B}^2}\frac{{\sin (\pi \mathbf{B}/2)}}{{\cos (\pi \mathbf{B}/2)}} = {(1 - \mathbf{B})^2}\frac{{\sin (\pi (1 - \mathbf{B})/2)}}{{\cos (\pi (1 - \mathbf{B})/2)}} = (1 - 2\mathbf{B} + {\mathbf{B}^2})\frac{{\cos (\pi \mathbf{B}/2)}}{{\sin (\pi \mathbf{B}/2)}},
\end{equation}
Using the Taylor expansion,we can prove the following equations,
\begin{equation}
\frac{{\cos (\pi \mathbf{B}/2)}}{{\sin (\pi \mathbf{B}/2)}} = \frac{2}{{\pi \mathbf{B}}} - \frac{4}{2} \cdot \frac{1}{6} \cdot \frac{{\pi \mathbf{B}}}{2} = \frac{2}{\pi } \cdot \frac{{{\pi ^2}}}{6} - \frac{\pi }{{12}} = \frac{\pi }{4},
\end{equation}

\begin{equation}
{\mathbf{B}^2}\frac{{\sin (\pi \mathbf{B}/2)}}{{\cos (\pi \mathbf{B}/2)}} = 0,
\end{equation}

\begin{equation}
 {\mathbf{B}^2}\frac{{\cos (\pi \mathbf{B}/2)}}{{\sin (\pi \mathbf{B}/2)}} = \frac{{2{\mathbf{B}^2}}}{{\pi \mathbf{B}}} = \frac{{2\mathbf{B}}}{\pi } = \frac{1}{\pi }.
\end{equation}
Using(4.21),(4.22),(4.23)and(4.24),yields
\begin{equation}
\frac{\pi }{2}\mathbf{B}\frac{{\cos (\pi \mathbf{B}/2)}}{{\sin (\pi \mathbf{B}/2)}} = \frac{{{\pi ^2}}}{{16}} + \frac{1}{4}.
\end{equation}
Let us now calculate the value of $ - \mathbf{B}\frac{{\Pi '( - \mathbf{B})}}{{\Pi ( - \mathbf{B})}}$  ,taking the logarithmic derivative on both sides of the equation $\frac{{\pi s}}{{\Pi (s)\Pi ( - s)}} = \sin \pi s$.
Let $s \to \mathbf{B}$,and both sides multiplied by $\mathbf{B}$,yields
\begin{equation}
\mathbf{B}\frac{{\Pi '(\mathbf{B})}}{{\Pi (\mathbf{B})}} - \mathbf{B}\frac{{\Pi '( - \mathbf{B})}}{{\Pi ( - \mathbf{B})}} = \frac{\mathbf{B}}{\mathbf{B}} - \frac{{\pi \mathbf{B}\cos \pi \mathbf{B}}}{{\sin \pi \mathbf{B}}} = 1 - \frac{{\pi \mathbf{B}\cos \pi \mathbf{B}}}{{\sin \pi \mathbf{B}}}.
\end{equation}
And now calculate the value of  $\frac{{\pi \mathbf{B}\cos \pi \mathbf{B}}}{{\sin \pi \mathbf{B}}}$.
Since $\frac{{\cos \pi \mathbf{B}}}{{\sin \pi \mathbf{B}}} = \frac{{\cos \pi (1 - \mathbf{B})}}{{\sin \pi (1 - \mathbf{B})}} =  - \frac{{\cos \pi \mathbf{B}}}{{\sin \pi \mathbf{B}}}$,we have
\[\frac{{\cos \pi \mathbf{B}}}{{\sin \pi \mathbf{B}}} = 0\]
and \[\pi {\mathbf{B}^2}\frac{{\cos \pi \mathbf{B}}}{{\sin \pi \mathbf{B}}} = \pi {(1 - \mathbf{B})^2}\frac{{\cos \pi (1 - \mathbf{B})}}{{\sin \pi (1 - \mathbf{B})}} =  - \pi (1 - 2\mathbf{B} + {\mathbf{B}^2})\frac{{\cos \pi \mathbf{B}}}{{\sin \pi \mathbf{B}}}.\]
Therefore
\[\pi {\mathbf{B}^2}\frac{{\cos \pi \mathbf{B}}}{{\sin \pi \mathbf{B}}} = \pi \mathbf{B}\frac{{\cos \pi \mathbf{B}}}{{\sin \pi \mathbf{B}}}.\]
Since $\pi {\mathbf{B}^2}\frac{{\cos \pi \mathbf{B}}}{{\sin \pi \mathbf{B}}} = \frac{{\pi {\mathbf{B}^2}}}{{\pi \mathbf{B}}} = \mathbf{B} = \frac{1}{2}$,we obtain
\begin{equation}
\pi \mathbf{B}\frac{{\cos \pi \mathbf{B}}}{{\sin \pi \mathbf{B}}} = \frac{1}{2}.
\end{equation}
Using(4.26)and(4.27),we obtain
\begin{equation}
\mathbf{B}\frac{{\Pi '(\mathbf{B})}}{{\Pi (\mathbf{B})}} - \mathbf{B}\frac{{\Pi '( - \mathbf{B})}}{{\Pi ( - \mathbf{B})}} = \frac{1}{2}.
\end{equation}
Using (4.2) of Lemma 4.2 , we obtain
\begin{equation}
\mathbf{B}\frac{{\Pi '(\mathbf{B})}}{{\Pi (\mathbf{B})}} =  - \mathbf{B}\frac{{\Pi '( - \mathbf{B})}}{{\Pi ( - \mathbf{B})}} + \frac{{\Pi '(0)}}{{\Pi (0)}} =  - \mathbf{B}\frac{{\Pi '( - \mathbf{B})}}{{\Pi ( - \mathbf{B})}} - \gamma,
\end{equation}
where $\gamma$ is Euler constant.Using(4.28)and(4.29),we obtain
\begin{equation}
\mathbf{B}\frac{{\Pi '(\mathbf{B})}}{{\Pi (\mathbf{B})}} = \frac{1}{4} - \frac{\gamma }{2}, - \mathbf{B}\frac{{\Pi '( - \mathbf{B})}}{{\Pi ( - \mathbf{B})}} = \frac{1}{4} + \frac{\gamma }{2}.
\end{equation}
Using(4.20),(4.25)and(4.30),we obtain
\[\frac{{\zeta '(\mathbf{B})}}{{\zeta (\mathbf{B})}} = \frac{1}{4} + \frac{\gamma }{2} + \mathbf{B}\log 2\pi  + \frac{{{\pi ^2}}}{{16}} + \frac{1}{4} = \frac{{1 + \gamma  + \log 2\pi }}{2} + \frac{{{\pi ^2}}}{{16}}.\]

\par And now calculate the value of $\sin \pi \mathbf{B}\cdot\frac{{\zeta '(\mathbf{B})}}{{\zeta (\mathbf{B})}}$ .Both sides of (4.19) multiplied by $\sin \pi \mathbf{B}$ , and let $s \to \mathbf{B}$,we obtain
\[\sin \pi \mathbf{B}\cdot\frac{{\zeta '(\mathbf{B})}}{{\zeta (\mathbf{B})}} =  - \sin \pi \mathbf{B}\cdot\frac{{\Pi '( - \mathbf{B})}}{{\Pi ( - \mathbf{B})}} + \log 2\pi  \cdot \sin \pi \mathbf{B} + \frac{\pi }{2}\frac{{\cos (\pi \mathbf{B}/2)}}{{\sin (\pi \mathbf{B}/2)}}\sin \pi \mathbf{B} - \sin \pi \mathbf{B}\cdot\frac{{\zeta '(1 - \mathbf{B})}}{{\zeta (1 - \mathbf{B})}}\]
\[ =  - \sin \pi \mathbf{B}\cdot\frac{{\Pi '( - \mathbf{B})}}{{\Pi ( - \mathbf{B})}} + \log 2\pi  \cdot \pi \mathbf{B} + \pi {\cos ^2}(\pi \mathbf{B}/2) - \sin \pi (1 - \mathbf{B})\cdot\frac{{\zeta '(\mathbf{B})}}{{\zeta (\mathbf{B})}}\]
\[ =  - \sin \pi \mathbf{B}\cdot\frac{{\Pi '( - \mathbf{B})}}{{\Pi ( - \mathbf{B})}} + \frac{{\pi \log 2\pi }}{2} + \pi \frac{{1 + \cos \pi \mathbf{B}}}{2} - \sin \pi \mathbf{B}\cdot\frac{{\zeta '(\mathbf{B})}}{{\zeta (\mathbf{B})}}.\]
Therefore
\begin{equation}
\sin \pi \mathbf{B}\cdot\frac{{\zeta '(\mathbf{B})}}{{\zeta (\mathbf{B})}} =  - \frac{{\sin \pi \mathbf{B}}}{2}\frac{{\Pi '( - \mathbf{B})}}{{\Pi ( - \mathbf{B})}} + \frac{{\pi \log 2\pi }}{4} + \frac{\pi }{4}.
\end{equation}
Firstly, we calculate the value of $\sin \pi \mathbf{B}\cdot\frac{{\Pi '( - \mathbf{B})}}{{\Pi ( - \mathbf{B})}}$.
Using (4.1) of Lemma 4.2, we have
\begin{equation}
\sin \pi \mathbf{B}\cdot\frac{{\Pi '(\mathbf{B})}}{{\Pi (\mathbf{B})}} = \sin ( - \pi \mathbf{B})\cdot\frac{{\Pi '( - \mathbf{B})}}{{\Pi ( - \mathbf{B})}} + \pi \frac{{\Pi '(0)}}{{\Pi (0)}} =  - \sin \pi \mathbf{B}\cdot\frac{{\Pi '( - \mathbf{B})}}{{\Pi ( - \mathbf{B})}} - \pi \gamma .
\end{equation}
Taking the logarithmic derivative on both sides of the equation $\frac{{\pi s}}{{\Pi (s)\Pi ( - s)}} = \sin \pi s$ .
Let $s \to \mathbf{B}$,and both sides multiplied by $\sin \pi \mathbf{B}$,we obtain
\[\sin \pi \mathbf{B}\cdot(\frac{{\Pi '(\mathbf{B})}}{{\Pi (\mathbf{B})}} - \frac{{\Pi '( - \mathbf{B})}}{{\Pi ( - \mathbf{B})}}) = \frac{{\sin \pi \mathbf{B}}}{\mathbf{B}} - \pi \sin \pi \mathbf{B}\frac{{\cos \pi \mathbf{B}}}{{\sin \pi \mathbf{B}}} = \frac{{\sin \pi \mathbf{B}}}{\mathbf{B}} - \pi \cos \pi \mathbf{B} = \frac{{\sin \pi \mathbf{B}}}{\mathbf{B}}.\]
We can prove the equation  $\frac{{\sin \pi \mathbf{B}}}{\mathbf{B}} = \pi \log 2$  by\eqref{1.1},
therefore
\begin{equation}
\sin \pi \mathbf{B}\cdot\frac{{\Pi '(\mathbf{B})}}{{\Pi (\mathbf{B})}} - \sin \pi \mathbf{B}\cdot\frac{{\Pi '( - \mathbf{B})}}{{\Pi ( - \mathbf{B})}} = \pi \log 2.
\end{equation}
Using(4.32)and(4.33),we obtain
\begin{equation}
\sin \pi \mathbf{B}\cdot\frac{{\Pi '( - \mathbf{B})}}{{\Pi ( - \mathbf{B})}} = \frac{{\pi \log 2 + \pi \gamma }}{{ - 2}}.
\end{equation}
Using(4.31)and(4.34),we obtain
\[\sin \pi \mathbf{B}\cdot\frac{{\zeta '(\mathbf{B})}}{{\zeta (\mathbf{B})}} = \frac{{\pi \log 2 + \pi \gamma }}{4} + \frac{{\pi \log 2\pi }}{4} + \frac{\pi }{4} = \frac{\pi }{4}(1 + \gamma  + \log 4\pi ).\]
\end{proof}

\begin{Proposition}
Let ${\lambda _1} = \sum\nolimits_{n = 1}^\infty  [{\sum\nolimits_{k = 1}^\infty  {\frac{1}{{{{(2n + k)}^2}}} - \frac{1}{{2n}}]} }$,${\lambda _2} = \sum\nolimits_{n = 1}^\infty  [{\sum\nolimits_{k = 1}^\infty  {\frac{{{{( - 1)}^{k - 1}}}}{{2n + k}} - \frac{1}{{4n}}]} } ,$  we have \[{\lambda _1} = \frac{{1 + \log 2}}{2} - \frac{5}{{48}}{\pi ^2},\] \[{\lambda _2} = \frac{{1 - 2\log 2}}{4}.\]
\end{Proposition}
\begin{proof}
We have the following equation (see \cite{E} p.158),
\begin{equation}
\frac{{\zeta '(s)}}{{\zeta (s)}} = \frac{{ - 1}}{{s - 1}} + \frac{{\gamma  + \log \pi }}{2} + \sum\limits_\rho  {\frac{1}{{s - \rho }}}  + \sum\limits_{n = 1}^\infty  {(\frac{1}{{s + 2n}} - \frac{1}{{2n}})} .
\end{equation}
where $\rho $ are zeros of $\xi (s)$.
Let $s \to \mathbf{B}$,we have
\begin{equation}
\frac{{\zeta '(\mathbf{B})}}{{\zeta (\mathbf{B})}} = \frac{{ - 1}}{{\mathbf{B} - 1}} + \frac{{\gamma  + \log \pi }}{2} + \sum\limits_\rho  {\frac{1}{{\mathbf{B} - \rho }}}  + \sum\limits_{n = 1}^\infty  {(\frac{1}{{\mathbf{B} + 2n}} - \frac{1}{{2n}})}.
\end{equation}
Since
\[\sum\limits_\rho  {\frac{1}{{\mathbf{B} - \rho }}}  = \sum\limits_\rho  {\frac{1}{{1 - \mathbf{B} - \rho }}}  = \sum\limits_\rho  {\frac{1}{{1 - \rho  - \mathbf{B}}}}  = \sum\limits_\rho  {\frac{1}{{\rho  - \mathbf{B}}}}, \]
we obtain
\begin{equation}
\sum\limits_\rho  {\frac{1}{{\rho  - \mathbf{B}}}}  = 0.
\end{equation}
Using(1.3),we have
\begin{equation}
{(\mathbf{B} + 2n)^{ - 1}} = (2 - 1)(\zeta (2) - {1^{ - 2}} - {2^{ - 2}} - {3^{ - 2}} - ... - {(2n)^{ - 2}}) = \sum\limits_{k = 1}^\infty  {\frac{1}{{{{(2n + k)}^2}}}}.
\end{equation}
Using(4.36),(4.37)and(4.38),we obtain
\begin{equation}
\frac{{\zeta '(\mathbf{B})}}{{\zeta (\mathbf{B})}} = \frac{{ - 1}}{{\mathbf{B} - 1}} + \frac{{\gamma  + \log \pi }}{2} + {\lambda _1} = \frac{1}{\mathbf{B}} + \frac{{\gamma  + \log \pi }}{2} + {\lambda _1} = \frac{{{\pi ^2}}}{6} + \frac{{\gamma  + \log \pi }}{2} + {\lambda _1}.
\end{equation}
By Proposition 4.1 and(4.39),we obtain
\[{\lambda _1} = \frac{{1 + \log 2}}{2} - \frac{5}{{48}}{\pi ^2}.\]

\par Both sides of (4.35) multiplied by $\sin \pi \mathbf{B}$ , and let $s \to \mathbf{B}$,we obtain
\[\sin \pi \mathbf{B}\cdot\frac{{\zeta '(\mathbf{B})}}{{\zeta (\mathbf{B})}} = \frac{{\sin \pi \mathbf{B}}}{{1 - \mathbf{B}}} + \frac{{\gamma  + \log \pi }}{2}\sin \pi \mathbf{B} + \sum\limits_\rho  {\frac{{\sin \pi \mathbf{B}}}{{\mathbf{B} - \rho }}}  + \sum\limits_{n = 1}^\infty  {\sin \pi \mathbf{B}\cdot(\frac{1}{{\mathbf{B} + 2n}} - \frac{1}{{2n}})} \]
\[ = \frac{{\sin \pi (1 - \mathbf{B})}}{\mathbf{B}} + \frac{{\gamma  + \log \pi }}{2}\pi \mathbf{B} + \sum\limits_\rho  {\frac{{\sin \pi \mathbf{B}}}{{\mathbf{B} - \rho }}}  + \sum\limits_{n = 1}^\infty  {(\frac{{\sin \pi \mathbf{B}}}{{\mathbf{B} + 2n}} - \frac{{\pi \mathbf{B}}}{{2n}})} \]
\[ = \frac{{\sin \pi \mathbf{B}}}{\mathbf{B}} + \frac{{\gamma  + \log \pi }}{4}\pi  + \sum\limits_\rho  {\frac{{\sin \pi \mathbf{B}}}{{\mathbf{B} - \rho }}}  + \sum\limits_{n = 1}^\infty  {(\frac{{\sin \pi \mathbf{B}}}{{\mathbf{B} + 2n}} - \frac{\pi }{{4n}})} \]
\begin{equation}
 = \pi \log 2 + \frac{{\gamma  + \log \pi }}{4}\pi  + \sum\limits_\rho  {\frac{{\sin \pi \mathbf{B}}}{{\mathbf{B} - \rho }}}  + \sum\limits_{n = 1}^\infty  {(\frac{{\sin \pi \mathbf{B}}}{{\mathbf{B} + 2n}} - \frac{\pi }{{4n}})} .
\end{equation}
Similarly,we can prove that
\begin{equation}
\sum\limits_\rho  {\frac{{\sin \pi \mathbf{B}}}{{\mathbf{B} - \rho }}}  = 0.
\end{equation}
Using\eqref{1.1},we have
\begin{equation}
\frac{{\sin \pi \mathbf{B}}}{{\mathbf{B} + 2n}} =  - \sum\limits_{k = 1}^\infty  {\frac{{\pi (2n + k)\cos k\pi  - \sin k\pi }}{{{{(2n + k)}^2}}}}  = \pi \sum\limits_{k = 1}^\infty  {\frac{{{{( - 1)}^{k - 1}}}}{{2n + k}}}.
\end{equation}
Using(4.40),(4.41)and(4.42),we obtain
\[\sin \pi \mathbf{B}\cdot\frac{{\zeta '(\mathbf{B})}}{{\zeta (\mathbf{B})}} = \pi \log 2 + \frac{{\gamma  + \log \pi }}{4}\pi  + \pi \sum\limits_{n = 1}^\infty  [{\sum\limits_{k = 1}^\infty  {\frac{{{{( - 1)}^{k - 1}}}}{{2n + k}}}  - \frac{1}{{4n}}]} \]
\begin{equation}
 = \pi \log 2 + \frac{{\gamma  + \log \pi }}{4}\pi  + \pi {\lambda _2}.
\end{equation}
By Proposition 4.1 and(4.43),we obtain
\[{\lambda _2} = \frac{{1 - 2\log 2}}{4}.\]
\end{proof}

\vspace{0.3in}
\section{\textbf{The functional equation of  $\log \Pi (\mathbf{B}s)$}}
\begin{Proposition}
The function $\log \Pi (\mathbf{B}s)$ satisfy the following functional equation
\begin{equation}
\log \Pi (\mathbf{B}s) = s\log \Pi (\frac{\mathbf{B}}{s}) + \frac{{s + 1}}{2}\log s + (1 - s)\frac{{\log 2\pi }}{2}.
\end{equation}
\end{Proposition}
\begin{proof}
Using the Stirling formula (see \cite{B} p.109), let $s \to \mathbf{B}/s$, we obtain
\begin{equation}
\log \Pi (\frac{\mathbf{B}}{s}) = (\frac{\mathbf{B}}{s} + \frac{1}{2})\log \frac{\mathbf{B}}{s} - \frac{\mathbf{B}}{s} + \frac{{\log 2\pi }}{2} + \sum\limits_{n = 1}^\infty  {\frac{{{{( - 1)}^{n + 1}}B_{2n}^2{{(2\pi )}^{2n}}}}{{2n \cdot 2 \cdot (2n)!}}} {s^{2n - 1}}.
\end{equation}
Note that we have used the following formula
\[{\mathbf{B}^{1 - 2k}} = \zeta (2k)(2k - 1) = \frac{{{{( - 1)}^{k + 1}}{{(2\pi )}^{2n}}{B_{2k}}}}{{2 \cdot (2k)!}},\]
where k is positive integer.Using(4.9),we have
\begin{equation}
\sum\limits_{n = 1}^\infty  {\frac{{{{( - 1)}^{n + 1}}B_{2n}^2{{(2\pi )}^{2n}}}}{{2n \cdot 2 \cdot (2n)!}}} {s^{2n - 1}} =  - \frac{{\log \sin \pi \mathbf{B}s}}{{2s}} + \frac{{\log \pi }}{{2s}} + \frac{{\log \mathbf{B}s}}{{2s}}.
\end{equation}
Using(5.2)and(5.3),we have
\begin{equation}
\log \Pi (\frac{\mathbf{B}}{s}) = (\frac{\mathbf{B}}{s} + \frac{1}{2})\log \frac{\mathbf{B}}{s} - \frac{\mathbf{B}}{s} + \frac{{\log 2\pi }}{2} - \frac{{\log \sin \pi \mathbf{B}s}}{{2s}} + \frac{{\log \pi }}{{2s}} + \frac{{\log \mathbf{B}s}}{{2s}}.
\end{equation}
By \[\frac{{\pi s}}{{\Pi (s)\Pi ( - s)}} = \sin \pi s,\]
we have
\begin{equation}
\log \Pi (\mathbf{B}s) + \log \Pi ( - \mathbf{B}s) = \log \pi \mathbf{B}s - \log \sin \pi \mathbf{B}s.
\end{equation}
Using (4.1) of Lemma 4.2 , we have
\[\log \Pi (\mathbf{B}s) = \log \Pi ( - \mathbf{B}s) + \frac{{\Pi '(0)}}{{\Pi (0)}}s.\]
Therefore
\begin{equation}
\log \Pi (\mathbf{B}s) = \log \Pi ( - \mathbf{B}s) - \gamma s.
\end{equation}
Using(5.5)and(5.6),we have
\begin{equation}
\log \Pi (\mathbf{B}s) = \frac{1}{2}( - \gamma s - \gamma  + \log \pi s - \log \sin \pi \mathbf{B}s).
\end{equation}
Using(5.4)and(5.7),we obtain(5.1).
\end{proof}
\begin{remark}
I find that Proposition 5.1 is equivalent to Ramanujan and A.P. Guinand's result,but the proofs are not the same(see \cite{A} ,\cite{H}).
\end{remark}

\vspace{0.3in}
\section{\textbf{The functional equation of $\log \zeta (\mathbf{B} + s)$ }}
\begin{theorem}
The function $\log \zeta (\mathbf{B} + s)$ satisfy the following functional equation
\begin{equation}
\log \zeta (\mathbf{B} - s) - \log \zeta (\mathbf{B} + s) = s\frac{{\Pi '(s)}}{{\Pi (s)}} - \frac{1}{2} - s - s\log 2\pi  + \frac{\pi }{2}s\frac{{\cos \pi s}}{{\sin \pi s}} + \frac{\pi }{4}(2\frac{{\cos \pi s}}{{\sin \pi s}} - \frac{{\cos (\pi s/2)}}{{\sin (\pi s/2)}}).
\end{equation}
Moreover the function $\log \zeta (\mathbf{B} + s)$ has poles at the non-positive integers(i.e. at $s =0, -1, -2, -3, ...$).
\end{theorem}
\begin{lemma}
\[\log (\mathbf{B} + s) = \frac{{\Pi '(s)}}{{\Pi (s)}}.\]
\end{lemma}
\begin{proof}
By $\Pi (s) = s\Pi (s - 1)$ ,we have
\[\log \Pi (\mathbf{B} + s) = \log (\mathbf{B} + s) + \log \Pi (\mathbf{B} + s - 1)\]
\begin{equation}
 = \log (\mathbf{B} + s) + \log \Pi ( - \mathbf{B} + s).
\end{equation}
Using(4.1),we have
\begin{equation}
\log \Pi (\mathbf{B} + s) = \log \Pi ( - \mathbf{B} + s) + \frac{{\Pi '(s)}}{{\Pi (s)}}.
\end{equation}
Using(6.2)and(6.3),we complete the proof.
\end{proof}
\begin{remark}
In particular,set s=0 in Lemma 6.2,we obtain \eqref{1.7} again.The function $\log (\mathbf{B} + s)$ has poles at the negative integers (i.e. at $s = -1, -2, -3, ...$).
\end{remark}
\begin{lemma}
\[(\mathbf{B} + s)\log (\mathbf{B} + s) - \mathbf{B}\log \mathbf{B} - s = \log \Pi (s).\]
\end{lemma}
\begin{proof}
By Lemma 6.2,we have
\[\int {\log (\mathbf{B} + s)} ds = \int {\frac{{\Pi '(s)}}{{\Pi (s)}}} ds.\]
Therefore
\[(\mathbf{B} + s)\log (\mathbf{B} + s) - s = \log \Pi (s) + c.\]
Let $s=0$,we obtain $c = \mathbf{B}\log \mathbf{B}$ .
\end{proof}
\begin{remark}
We can obtain a very short proof of Stirling formula by Lemma 6.4.
\begin{proof}
Using Lemma 6.4,we have
\[\log \Pi (s) = (\mathbf{B} + s)\log (1 + \frac{\mathbf{B}}{s}) + (\mathbf{B} + s)\log s - s - \mathbf{B}\log \mathbf{B}\]
\[ = (\mathbf{B} + s)\sum\limits_{n = 1}^\infty  {{{( - 1)}^{n - 1}}\frac{{{\mathbf{B}^n}}}{{n{s^n}}} + (\frac{1}{2} + s)\log s}  - s - \frac{{1 - \log 2\pi }}{2}.\]
\end{proof}
\end{remark}
\begin{remark}
We can obtain another proof of Proposition 5.1 by Lemma 6.4.
\begin{proof}
Set $s \to \mathbf{B}'/s$ in Lemma 6.4,where $\mathbf{B}'$ is another Bernoulli operator,we have
\[\log \Pi (\mathbf{B}'/s) = (\mathbf{B} + \mathbf{B}'/s)\log (\mathbf{B} + \mathbf{B}'/s) - \mathbf{B}\log \mathbf{B} - \mathbf{B}'/s\]
\[ = (\mathbf{B} + \mathbf{B}'/s)\log (\mathbf{B}s + \mathbf{B}') - (\mathbf{B} + \mathbf{B}'/s)\log s - \mathbf{B}\log \mathbf{B} - \mathbf{B}'/s.\]
Therefore
 \[s\log \Pi (\mathbf{B}'/s) = (\mathbf{B}s + \mathbf{B}')\log (\mathbf{B}s + \mathbf{B}') - (\mathbf{B}s + \mathbf{B}')\log s - \mathbf{B}\log \mathbf{B} \cdot s - \mathbf{B}'.\]
On the other hand,set $s \to \mathbf{B}'s$ in Lemma 6.4,where $\mathbf{B}'$ is another Bernoulli operator,we have
\[\log \Pi (\mathbf{B}'s) = (\mathbf{B} + \mathbf{B}'s)\log (\mathbf{B} + \mathbf{B}'s) - \mathbf{B}\log \mathbf{B} - \mathbf{B}'s.\]
Since \[(\mathbf{B}s + \mathbf{B}')\log (\mathbf{B}s + \mathbf{B}') = (\mathbf{B} + \mathbf{B}'s)\log (\mathbf{B} + \mathbf{B}'s),\] we have
\[s\log \Pi (\mathbf{B}'/s) = (\mathbf{B} + \mathbf{B}'s)\log (\mathbf{B} + \mathbf{B}'s) - \frac{{1 + s}}{2}\log s - \frac{{1 - \log 2\pi }}{2}s - \frac{1}{2}\]
\[ = \log \Pi (\mathbf{B}'s) + \frac{{1 - \log 2\pi }}{2} + \frac{s}{2} - \frac{{1 + s}}{2}\log s - \frac{{1 - \log 2\pi }}{2}s - \frac{1}{2}\]
\[ = \log \Pi (\mathbf{B}'s) + \frac{{\log 2\pi }}{2}s - \frac{{\log 2\pi }}{2} - \frac{{1 + s}}{2}\log s.\]
\end{proof}
\end{remark}

\begin{lemma}
\[\log \Pi (s - \mathbf{B}) = s\frac{{\Pi '(s)}}{{\Pi (s)}} - \frac{{1 - \log 2\pi }}{2} - s.\]
\end{lemma}
\begin{proof}
By $\Pi (s) = s\Pi (s - 1)$ ,we have
\[\mathbf{B}\log \Pi (\mathbf{B} + s) = \mathbf{B}\log (\mathbf{B} + s) + \mathbf{B}\log \Pi (\mathbf{B} + s - 1)\]
\begin{equation}
= \mathbf{B}\log (\mathbf{B} + s) + (1 - \mathbf{B})\log \Pi ( - \mathbf{B} + s).
\end{equation}
Using (4.2),we have
\begin{equation}
\mathbf{B}\log \Pi (\mathbf{B} + s) =  - \mathbf{B}\log \Pi ( - \mathbf{B} + s) + \log \Pi (s).
\end{equation}
Combining(6.4)and(6.5),we obtain
\[\log \Pi (s) = \mathbf{B}\log (\mathbf{B} + s) + \log \Pi (s - \mathbf{B}).\]
By Lemma 6.2 and Lemma 6.4,we have
\[\log \Pi (s) = \log \Pi (s) + \mathbf{B}\log \mathbf{B} + s - s\log (\mathbf{B} + s) + \log \Pi (s - \mathbf{B})\]
\[ = \log \Pi (s) + \frac{{1 - \log 2\pi }}{2} + s - s\frac{{\Pi '(s)}}{{\Pi (s)}} + \log \Pi (s - \mathbf{B}).\]
\end{proof}

\begin{lemma}
\[\log \Pi (\frac{{\mathbf{B} + s}}{2}) = \frac{{\Pi '(s/2)}}{{4\Pi (s/2)}} - \frac{s}{2}\log 2 + \frac{s}{2}\frac{{\Pi '(s)}}{{\Pi (s)}} - \frac{s}{2} + \frac{{\Pi '(s)}}{{2\Pi (s)}} - \frac{1}{4} + \frac{{\log \pi }}{2}.\]
\end{lemma}
\begin{proof}
By \[\Pi (s) = {2^s}\Pi (\frac{s}{2})\Pi (\frac{{s - 1}}{2}){\pi ^{ - 1/2}},\]
we have
\[\log \Pi (\mathbf{B} + s) = (\mathbf{B} + s)\log 2 + \log \Pi (\frac{{\mathbf{B} + s}}{2}) + \log \Pi (\frac{{s - \mathbf{B}}}{2}) - \frac{1}{2}\log \pi \]
\[ = (\frac{1}{2} + s)\log 2 + \log \Pi (\frac{{\mathbf{B} + s}}{2}) + \log \Pi (\frac{{s + \mathbf{B}}}{2}) - \frac{1}{2}\frac{{\Pi '(s/2)}}{{\Pi (s/2)}} - \frac{1}{2}\log \pi. \]
Therefore
\[2\log \Pi (\frac{{\mathbf{B} + s}}{2}) = \log \Pi (\mathbf{B} + s) - (\frac{1}{2} + s)\log 2 + \frac{1}{2}\frac{{\Pi '(s/2)}}{{\Pi (s/2)}} + \frac{1}{2}\log \pi \]
\[ = \log \Pi ( - \mathbf{B} + s) + \frac{{\Pi '(s)}}{{\Pi (s)}} - (\frac{1}{2} + s)\log 2 + \frac{1}{2}\frac{{\Pi '(s/2)}}{{\Pi (s/2)}} + \frac{1}{2}\log \pi .\]
By Lemma 6.7,we obtain
\[2\log \Pi (\frac{{\mathbf{B} + s}}{2}) = s\frac{{\Pi '(s)}}{{\Pi (s)}} - \frac{{1 - \log 2\pi }}{2} - s + \frac{{\Pi '(s)}}{{\Pi (s)}} - (\frac{1}{2} + s)\log 2 + \frac{1}{2}\frac{{\Pi '(s/2)}}{{\Pi (s/2)}} + \frac{1}{2}\log \pi \]
\[ = (s + 1)\frac{{\Pi '(s)}}{{\Pi (s)}} - \frac{1}{2} + \log \pi  - s - s\log 2 + \frac{1}{2}\frac{{\Pi '(s/2)}}{{\Pi (s/2)}}.\]
\end{proof}

\begin{lemma}
\begin{equation}
\log \sin \frac{{\pi (\mathbf{B} - s)}}{2} =  - \log 2 + \frac{{\pi s}}{2}\frac{{\cos \pi s}}{{\sin \pi s}} + \frac{\pi }{4}(2\frac{{\cos \pi s}}{{\sin \pi s}} - \frac{{\cos (\pi s/2)}}{{\sin (\pi s/2)}}).
\end{equation}
\end{lemma}
\begin{proof}
Since \[\frac{{\pi s}}{{\Pi (s)\Pi ( - s)}} = \sin \pi s,\]
we have
\[\log \sin \frac{{\pi (\mathbf{B} - s)}}{2} = \log \frac{{\pi (\mathbf{B} - s)}}{2} - \log \Pi (\frac{{\mathbf{B} - s}}{2}) - \log \Pi (\frac{{ - \mathbf{B} + s}}{2})\]
\[ = \log \frac{\pi }{2} + \frac{{\Pi '( - s)}}{{\Pi ( - s)}} - \log \Pi (\frac{{\mathbf{B} - s}}{2}) - \log \Pi (\frac{{\mathbf{B} + s}}{2}) + \frac{1}{2}\frac{{\Pi '(s/2)}}{{\Pi (s/2)}}.\]
By Lemma 6.8,we obtain
\[\log \sin \frac{{\pi (\mathbf{B} - s)}}{2} = \log \frac{\pi }{2} + \frac{{\Pi '( - s)}}{{\Pi ( - s)}} - \frac{{\Pi '( - s/2)}}{{4\Pi ( - s/2)}} - \frac{s}{2}\log 2 + \frac{s}{2}\frac{{\Pi '( - s)}}{{\Pi ( - s)}} - \frac{s}{2} - \frac{1}{2}\frac{{\Pi '( - s)}}{{\Pi ( - s)}}\]
\[ + \frac{1}{4} - \frac{1}{2}\log \pi  - \frac{1}{4}\frac{{\Pi '(s/2)}}{{\Pi (s/2)}} + \frac{s}{2}\log 2 - \frac{s}{2}\frac{{\Pi '(s)}}{{\Pi (s)}} + \frac{s}{2} - \frac{1}{2}\frac{{\Pi '(s)}}{{\Pi (s)}} + \frac{1}{4} - \frac{{\log \pi }}{2} + \frac{1}{2}\frac{{\Pi '(s/2)}}{{\Pi (s/2)}}\]
\[ = \log \frac{\pi }{2} + \frac{{\Pi '( - s)}}{{\Pi ( - s)}} - \frac{1}{4}[\frac{{\Pi '(s/2)}}{{\Pi (s/2)}} + \frac{{\Pi '( - s/2)}}{{\Pi ( - s/2)}}] + \frac{s}{2}[\frac{{\Pi '( - s)}}{{\Pi ( - s)}} - \frac{{\Pi '(s)}}{{\Pi (s)}}]\]
\begin{equation}
 - \frac{1}{2}[\frac{{\Pi '(s)}}{{\Pi (s)}} + \frac{{\Pi '( - s)}}{{\Pi ( - s)}}] + \frac{1}{2} - \log \pi  + \frac{1}{2}\frac{{\Pi '(s/2)}}{{\Pi (s/2)}}.
\end{equation}
Moreover
\begin{equation}
\frac{{\Pi '(s/2)}}{{2\Pi (s/2)}} - \frac{{\Pi '( - s/2)}}{{2\Pi ( - s/2)}} = \frac{1}{s} - \frac{\pi }{2}\frac{{\cos (\pi s/2)}}{{\sin (\pi s/2)}},
\end{equation}
\begin{equation}
\frac{{\Pi '(s)}}{{\Pi (s)}} - \frac{{\Pi '( - s)}}{{\Pi ( - s)}} = \frac{1}{s} - \pi \frac{{\cos \pi s}}{{\sin \pi s}}.
\end{equation}
Using(6.7),(6.8)and(6.9),we obtain(6.6).
\end{proof}
\par We now prove the Theorem 6.1.
\begin{proof}
Taking the logarithms on both sides of functional equation of Riemann Zeta, let $s \to \mathbf{B} - s$,we have
\[\log \zeta (\mathbf{B} - s) = \log \Pi (s - \mathbf{B}) + \log 2\pi  \cdot ( - s + \mathbf{B} - 1) + \log 2 + \log \sin \frac{{\pi (\mathbf{B} - s)}}{2} + \log \zeta (1 - \mathbf{B} + s).\]
By lemma 6.7 and lemma 6.9,we have
\[\log \zeta (\mathbf{B} - s) = s\frac{{\Pi '(s)}}{{\Pi (s)}} - \frac{{1 - \log 2\pi }}{2} - s - (\frac{1}{2} + s)\log 2\pi  + \log 2 - \log 2 + \frac{{\pi s}}{2}\frac{{\cos \pi s}}{{\sin \pi s}}\]
\[ + \frac{\pi }{4}(2\frac{{\cos \pi s}}{{\sin \pi s}} - \frac{{\cos (\pi s/2)}}{{\sin (\pi s/2)}}) + \log \zeta (\mathbf{B} + s)\]
\[ = s\frac{{\Pi '(s)}}{{\Pi (s)}} - \frac{1}{2} - s - s\log 2\pi  + \frac{{\pi s}}{2}\frac{{\cos \pi s}}{{\sin \pi s}} + \frac{\pi }{4}(2\frac{{\cos \pi s}}{{\sin \pi s}} - \frac{{\cos (\pi s/2)}}{{\sin (\pi s/2)}}) + \log \zeta (\mathbf{B} + s).\]
Because $\log \zeta (\mathbf{B} + s)$ is analytic at ${\mathop{\rm Re}\nolimits} {\kern 1pt} s > 0 $ ,using the functional equation (6.1),it is clear that the function $\log \zeta (\mathbf{B} + s)$ has poles at the negative integers (i.e. at $s = -1, -2, -3, ...$).Because \[\log \zeta (\mathbf{B}) = \log [\zeta (\mathbf{B})(\mathbf{B} - 1)] - \log (\mathbf{B} - 1),\]by \eqref{1.1},we have \[\log \zeta (\mathbf{B}) =  - \mathop {\lim }\limits_{s \to 1} \frac{{\zeta '(s)(s - 1) + \zeta (s)}}{{\zeta (s)(s - 1)}} - \sum\limits_{n = 2}^\infty  {[\frac{{\zeta '(n)}}{{\zeta (n)}}}  + \frac{1}{{n - 1}}] + \mathop {\lim }\limits_{s \to 1} \frac{1}{{s - 1}} + \sum\limits_{n = 2}^\infty  {\frac{1}{{n - 1}}} \]
\[ =  - \gamma  - \sum\limits_{n = 2}^\infty  {\frac{{\zeta '(n)}}{{\zeta (n)}}}  + \mathop {\lim }\limits_{s \to 1} \frac{1}{{s - 1}},\]
therefore the function $\log \zeta (\mathbf{B} + s)$ has a pole at $s=0$.By the way,we give \[\begin{array}{l}
 \log [ - \zeta (\mathbf{B})] = \log [\zeta (\mathbf{B})(\mathbf{B} - 1)] - \log (1 - \mathbf{B}) \\
  =  - \mathop {\lim }\limits_{s \to 1} \frac{{\zeta '(s)(s - 1) + \zeta (s)}}{{\zeta (s)(s - 1)}} - \sum\limits_{n = 2}^\infty  {[\frac{{\zeta '(n)}}{{\zeta (n)}} + \frac{1}{{n - 1}}]}  - \log \mathbf{B} \\
  =  - \gamma  - \sum\limits_{n = 2}^\infty  {\frac{{\zeta '(n)}}{{\zeta (n)}}}  + \log \mathbf{B} - \log \mathbf{B} \\
  =  - \gamma  - \sum\limits_{n = 2}^\infty  {\frac{{\zeta '(n)}}{{\zeta (n)}}} , \\
 \end{array}\]
therefore
\[\begin{array}{l}
 \log [ - \zeta (\mathbf{B})] =  - \gamma  - \sum\limits_{n = 2}^\infty  {\frac{{\zeta '(n)}}{{\zeta (n)}}}  =  - \gamma  + \log \zeta (1 + \mathbf{B}) \\
  =  - \gamma  + \sum\limits_{n = 2}^\infty  {\Lambda (n)} {(\log n)^{ - 1}}{n^{ - 1 - \mathbf{B}}} =  - \gamma  + \sum\limits_{n = 2}^\infty  {\Lambda (n)} {(\log n)^{ - 1}}{n^{ - 1}}\frac{{\log n}}{{n - 1}} \\
  =  - \gamma  + \sum\limits_{n = 2}^\infty  {\frac{{\Lambda (n)}}{{n(n - 1)}}}  \approx \log [ - \zeta (0)] + \frac{{\zeta '(0)}}{{\zeta (0)}}{B_1} =  - \log 2 + \frac{{\log 2\pi }}{2} \approx 0.226 , \\
 \end{array}\]
 where $\Lambda (n)$ is von Mangoldt function.By calculating, we can find the value of $\log [ - \zeta (\mathbf{B})]$ is close to the value of $\log [ - \zeta (1/2)]$.
\end{proof}
\begin{remark}
Set $s = 1$ in $(3.2)$,we have
\[\begin{array}{l}
 \log \zeta (\mathbf{B} + 1) =  - \log \Pi (\frac{{\mathbf{B} + 1}}{2}) + \frac{{\mathbf{B} + 1}}{2}\log \pi - \log (1 + \mathbf{B} - 1) + \log \xi (\mathbf{B}) + \sum\limits_\rho  {\log (1 - \frac{{1}}{{\rho  - \mathbf{B}}})}  \\
 \end{array}.\]
 On the other hand,\[ \log \xi (\mathbf{B})=\log \Pi (\frac{\mathbf{B}}{2}) + \log (1 - \mathbf{B}) - \frac{\mathbf{B}}{2}\log \pi  + \log [ - \zeta (\mathbf{B})].\]
 Therefore
 \[\begin{array}{l}
 \log \zeta (\mathbf{B} + 1) =  - \log \Pi (\frac{{\mathbf{B} + 1}}{2}) + \frac{{\mathbf{B} + 1}}{2}\log \pi- \log \mathbf{B} + \log \Pi (\frac{\mathbf{B}}{2}) \\
 + \log (1 - \mathbf{B}) - \frac{\mathbf{B}}{2}\log \pi  + \log [ - \zeta (\mathbf{B})] + \sum\limits_\rho  {\log (1 - \frac{{1}}{{\rho  - \mathbf{B}}})}  \\
 \end{array}\]
  \[\begin{array}{l}
  =  - \log \Pi (\frac{{\mathbf{B} + 1}}{2}) + \frac{{1}}{2}\log \pi  + \log \Pi (\frac{\mathbf{B}}{2}) + \log [ - \zeta (\mathbf{B})] + \sum\limits_\rho  {\log \frac{{\mathbf{B} + 1 - \rho }}{{\mathbf{B} - \rho }}}  \\
 \end{array}\]
 \[\begin{array}{l}
  =  - \log \Pi (\frac{{\mathbf{B} + 1}}{2}) + \frac{{1}}{2}\log \pi  + \log \Pi (\frac{\mathbf{B}}{2})  - \gamma  + \log \zeta (1 + \mathbf{B})  + \sum\limits_\rho  {\log \frac{{\mathbf{B} + 1 - \rho }}{{\mathbf{B} - \rho }}}  \\
 \end{array}.\]
 By Lemma 6.8,we have
\[\log \Pi (\frac{{\mathbf{B} + 1}}{2}) = \frac{{\Pi '(1/2)}}{{4\Pi (1/2)}} - \frac{1}{2}\log 2 + \frac{{\Pi '(1)}}{{\Pi (1)}} - \frac{1}{2} - \frac{1}{4} + \frac{{\log \pi }}{2},\]
\[\log \Pi (\frac{\mathbf{B}}{2}) = \frac{{\Pi '(0)}}{{4\Pi (0)}} + \frac{{\Pi '(0)}}{{2\Pi (0)}} - \frac{1}{4} + \frac{{\log \pi }}{2}.\]
By Lemma 6.2,we have
\[\sum\limits_\rho  {\log \frac{{\mathbf{B} + 1 - \rho }}{{\mathbf{B} - \rho }}}  = \sum\limits_\rho  {\frac{{\Pi '(1 - \rho )}}{{\Pi (1 - \rho )}}}  - \frac{{\Pi '( - \rho )}}{{\Pi ( - \rho )}} =  \sum\limits_\rho  {\frac{1}{1 - \rho }} = \sum\limits_\rho  {\frac{1}{\rho }}.\]
 Therefore we deduce that
\[\begin{array}{l}
0 =  - \frac{{\Pi '(1/2)}}{{4\Pi (1/2)}} + \frac{1}{2}\log 2 - \frac{{\Pi '(1)}}{{\Pi (1)}} + \frac{1}{2} + \frac{1}{4} - \frac{{\log \pi }}{2} + \frac{{\log \pi }}{2}\\
 + \frac{{\Pi '(0)}}{{4\Pi (0)}} + \frac{{\Pi '(0)}}{{2\Pi (0)}} - \frac{1}{4} + \frac{{\log \pi }}{2} - \gamma  + \sum\limits_\rho  {\frac{1}{\rho }},
\end{array}\]
we arrive at
\[\begin{array}{l}
0 =  - \frac{1}{4}(2 - 2\log 2 - \gamma ) + \frac{{\log 2}}{2} - 1 + \gamma  + \frac{1}{2}\\
 - \frac{\gamma }{4} - \frac{\gamma }{2} + \frac{{\log \pi }}{2} - \gamma  + \sum\limits_\rho  {\frac{1}{\rho }}.
\end{array}\]
Hence, we deduce that
\[\sum\limits_\rho  {\frac{1}{\rho }}  = 1 + \frac{\gamma }{2} - \frac{{\log \pi }}{2} - \log 2.\]
\end{remark}

\begin{Corollary}
\[ - s\frac{{\zeta '(\mathbf{B})}}{{\zeta (\mathbf{B})}} = \mathbf{B}s\frac{{\Pi '(\mathbf{B}s)}}{{\Pi (\mathbf{B}s)}} - \frac{{1 + s + s\log 2\pi }}{2} + \frac{{\mathbf{B}\pi s}}{2}\frac{{\cos \pi \mathbf{B}s}}{{\sin \pi \mathbf{B}s}} - \frac{{{\pi ^2}}}{{16}}s.\]
\end{Corollary}
\begin{proof}
Set $s \to \mathbf{B}'s$ ($\mathbf{B}'$  is also a Bernoulli operator) in (6.1),we have
\[\log \zeta (\mathbf{B} - \mathbf{B}'s) - \log \zeta (\mathbf{B} + \mathbf{B}'s) = \mathbf{B}'s\frac{{\Pi '(\mathbf{B}'s)}}{{\Pi (\mathbf{B}'s)}} - \frac{1}{2} - \mathbf{B}'s - \mathbf{B}'s\log 2\pi \]
\[ + \frac{{\pi \mathbf{B}'s}}{2}\frac{{\cos \pi \mathbf{B}'s}}{{\sin \pi \mathbf{B}'s}} + \frac{\pi }{4}(2\frac{{\cos \pi \mathbf{B}'s}}{{\sin \pi \mathbf{B}'s}} - \frac{{\cos (\pi \mathbf{B}'s/2)}}{{\sin (\pi \mathbf{B}'s/2)}})\]
\begin{equation}
= \mathbf{B}s\frac{{\Pi '(\mathbf{B}s)}}{{\Pi (\mathbf{B}s)}} - \frac{1}{2} - \mathbf{B}s - \mathbf{B}s\log 2\pi  + \frac{{\pi \mathbf{B}s}}{2}\frac{{\cos \pi \mathbf{B}s}}{{\sin \pi \mathbf{B}s}} + \frac{\pi }{4}(2\frac{{\cos \pi \mathbf{B}s}}{{\sin \pi \mathbf{B}s}} - \frac{{\cos (\pi \mathbf{B}s/2)}}{{\sin (\pi \mathbf{B}s/2)}}).
\end{equation}
Using(4.1),we have
\begin{equation}
\log \zeta (\mathbf{B} - \mathbf{B}'s) - \log \zeta (\mathbf{B} + \mathbf{B}'s) = \frac{{\zeta '(\mathbf{B})}}{{\zeta (\mathbf{B})}}s.
\end{equation}
Applying the Taylor expansion,we find that
\begin{equation}
\frac{\pi }{4}(2\frac{{\cos \pi \mathbf{B}s}}{{\sin \pi \mathbf{B}s}} - \frac{{\cos (\pi \mathbf{B}s/2)}}{{\sin (\pi \mathbf{B}s/2)}}) = \frac{\pi }{4}[2(\frac{1}{{\pi \mathbf{B}s}} - \frac{{{2^2}}}{2}{B_2}\pi \mathbf{B}s) - \frac{2}{{\pi \mathbf{B}s}} + \frac{{{2^2}}}{2}{B_2}\pi \frac{{\mathbf{B}s}}{2}] =  - \frac{{{\pi ^2}}}{{16}}s.
\end{equation}
Using(6.10),(6.11)and(6.12),we obtain Corollary 6.11.
\end{proof}
\par In particular,set $s=1$ in Corollary 6.11,combining (4.27) and (4.30),we obtain the following equation again
\[\frac{{\zeta '(\mathbf{B})}}{{\zeta (\mathbf{B})}} = \frac{{1 + \gamma  + \log 2\pi }}{2} + \frac{{{\pi ^2}}}{{16}}.\]

\vspace{0.3in}
\section{\textbf{On Riemann Hypothesis}}
If Riemann Hypothesis is true,then $\log \zeta (1/2 + s)$ is analytic at all points of ${\mathop{\rm Re}\nolimits} {\kern 1pt} s > 0$ except for $s=1/2$.On the other hand,the Bernoulli operator function $\log \zeta (\mathbf{B} + s)$ is an analogue of $\log \zeta (1/2 + s)$.Because $\log \zeta (\mathbf{B} + s)$  is analytic at ${\mathop{\rm Re}\nolimits} {\kern 1pt} s > 0 ,$ this seems to  believe that the Riemann hypothesis is true.In fact,if we can prove function $\log \zeta (\mathbf{B} + s)$  has singularities at $\rho  - \frac{1}{2}$ (where $\rho $ are zeros of $\xi (s)$),then Riemann Hypothesis is proved.
Set $s \to \rho ' - \frac{1}{2}$ in $(3.2)$,where $\rho '$ is a zero of $\xi (s)$,we have
\[\begin{array}{l}
 \log \zeta (\mathbf{B} + \rho ' - \frac{1}{2}) =  - \log \Pi (\frac{{\mathbf{B} + \rho ' - 1/2}}{2}) + \frac{{\mathbf{B} + \rho ' - 1/2}}{2}\log \pi  \\
  - \log (\rho ' - \frac{1}{2} + \mathbf{B} - 1) + \log \xi (\mathbf{B}) + \sum\limits_\rho  {\log (1 - \frac{{\rho ' - 1/2}}{{\rho  - \mathbf{B}}})}  \\
 \end{array}\]
 \[\begin{array}{l}
  =  - \log \Pi (\frac{{\mathbf{B} + \rho ' - 1/2}}{2}) + \frac{{\mathbf{B} + \rho ' - 1/2}}{2}\log \pi  \\
  - \log (\rho ' - \frac{1}{2} + \mathbf{B} - 1) + \log \xi (\mathbf{B}) + \sum\limits_\rho  {\log \frac{{\rho - \mathbf{B} - \rho ' + 1/2 }}{{ \rho - \mathbf{B}}}}   \\
 \end{array}.\]
 when $\rho  = \rho '$,the right of above equation will lead to a singular item $\log (1/2 - \mathbf{B})$ (But need a rigorous proof).
 Hence ${\mathop{\rm Re}\nolimits} {\kern 1pt} \rho ' = 1/2.$

\vspace{0.3in}
\bibliographystyle{amsplain}

\end{document}
